\documentclass[12pt]{iopart}
\usepackage{graphicx}
\expandafter\let\csname equation*\endcsname\relax
\expandafter\let\csname endequation*\endcsname\relax
\usepackage{amsmath,amsfonts,amsthm,amssymb}
\usepackage{bm}
\usepackage{booktabs}

\usepackage{epstopdf}
\usepackage{float}
\usepackage{color}
\usepackage{cases}  
\usepackage{subfigure}
\usepackage{algpseudocode}
\usepackage{caption}

\usepackage{multirow}
\usepackage{bigstrut}
\usepackage{longtable}
\usepackage{rotating}
\usepackage{esint,algorithm}
\usepackage{cite}
\usepackage{tablefootnote}
\usepackage[numbers]{natbib}
\usepackage{subcaption}

\numberwithin{equation}{section}
\allowdisplaybreaks[4]

\newtheorem{theorem}{Theorem}
\newtheorem{lemma}{Lemma}
\newtheorem{remark}{Remark}
\newtheorem{proposition}{Proposition}

\usepackage[top=0.8in,bottom=0.5in,left=0.85in,right=1.05in]{geometry}
\allowdisplaybreaks[4] 

\def\diint{\displaystyle\iint}

\newtheorem{assumption}{Assumption}
\newcommand{\newblock}{\hskip .11em\@plus.33em\@minus.07em}  

\begin{document}
	\title[Inexact ADMM for Parabolic Control]{An Inexact Alternating Direction Method of Multipliers for Constrained Parabolic Optimal Distributed Control Problems}

	\author{Haiming Song$^{1}$,
		Jinda Yang$^{2}$, 
		Yuran Yang$^{1}$, and
		Jianhua Yuan$^{3,4}$}
	
	\address{%
		$^{1}$ School of Mathematics, Jilin University,
		Changchun 130012, China\\
		$^{2}$ School of Mathematics and Statistics, Lanzhou University,
		Lanzhou 730000, China\\
		$^{3}$ School of Science, Beijing University of Post and Telecommunications,
		Beijing 100876, China\\
		$^{4}$ Key Laboratory of Mathematics and Information Networks
		(Beijing University of Posts and Telecommunications),
		Ministry of Education, Beijing 100876, China
	}
	
	\ead{songhaiming@jlu.edu.cn}
	\ead{yangyr24@mails.jlu.edu.cn}
	\ead{yangjinda@lzu.edu.cn}
	\ead{jianhuayuan@bupt.edu.cn}
	\date{\today}  

	
	\begin{abstract}
		Solving parabolic optimal control problems can be inherently challenging in the field of science and engineering, especially with constraints on the nonsmooth distributed  control. Motivated by the extensive applicability of the alternating direction method of multipliers, in this paper we develop a novel inexact algorithmic framework for parabolic optimal distributed control problems with control constraints. By decoupling the control constraint and possible nonsmooth objective from the optimal control problem,  our aim is to efficiently solve the subproblem constrained by the parabolic state equation, for which computing a sufficiently accurate numerical solution can be prohibitively expensive. Given this high computational cost, we consider that it may not always be justifiable to compute a highly accurate solution of the subproblem at every iteration. Hence, we propose an inexact strategy for solving the parabolic equation constrained subproblem. Under mild and flexible conditions on the parameters, we prove global convergence and a linear convergence rate for the resulting algorithmic framework. In practice, our inexact algorithmic framework is easily implementable with applicable nested iterations. Numerical experiments are performed on different cases, and the results demonstrate the validity and trustworthy performance of the proposed methods.

		
		\maketitle
	\end{abstract}

	

	\section{Introduction}\label{sec:introduction}
	Optimal control of partial differential equations (PDEs) has important applications in various scientific and industrial fields, such as physics, chemistry, medicine, finance, and engineering. We refer to, for instance, \cite{colli2022optimal,garcke2021sparse,sprekels2021sparse,bersetche2025fractional,wang2025adaptive,chang2022sparse,baraldi2022proximal} and the references therein. In particular, constrained parabolic optimal distributed control problems play a crucial role. In such problems, the state is governed by a linear parabolic equation, and the distributed control enters as a source term subject to additional constraints. Specifically, the control is allowed to be sparse and may act either on the whole spatial domain or only on a prescribed subdomain.
	
	From an algorithmic perspective, we are facing nonsmooth optimization problems that couple state equation constraint with additional control constraints, and that after numerical discretization, typically lead to large and complicated algebraic systems, especially in the context of time-dependent PDE-constrained problems. Numerically, such problems are better solved by algorithms that carefully exploit the structure and properties of the underlying model, rather than by generic off-the-shelf methods.

	In this paper, we consider the following constrained parabolic optimal distributed control problem:
	\begin{equation}\label{eq:obj}
		\min\limits_{u}\ \mathcal{J}(u)=\dfrac{\gamma_{d}}{2}\diint_{Q}|y(u)-y_{d}|^{2}{\rm d}x{\rm d}t+\dfrac{1}{2}\diint_{G}|u|^{2}{\rm d}x{\rm d}t+\gamma_{s} \diint_{G}|u|{\rm d}x{\rm d}t ,
	\end{equation}
	where $Q=\Omega\times(0,T)$ and $G=\Omega_{\text{sub}}\times(0,T)$, with $\Omega_{\text{sub}}$ an open subset of the spatial domain $\Omega\subset\mathbb{R}^{d}(d\geq 1)$, and $0<T<\infty$ a fixed final time. The positive constant $\gamma_{d}$ is the weight of the tracking term, the nonnegative constant $ \gamma_{s}$ is the sparsity-promoting parameter, and the desired state $y_{d}$ is given in $L^{2}(Q)$. Moreover, the state variable $y(u)$ corresponding to a control variable $u$ is governed by the parabolic equation:
	\begin{equation}\label{eq:state equation}
		\left\{
		\begin{aligned}
			&\frac{\partial y}{\partial t}-\nu\Delta y+a_{0}y=\chi_{G} u, &&\text{\,in}\quad\Omega\times(0,T),\\
			&y=0, &&\text{on}\quad\partial\Omega\times (0,T),\\
			&y(0)=\phi,&&\text{\,in}\quad\Omega.
		\end{aligned}
		\right.
	\end{equation}
	Here $\chi_{G}$ denotes the indicator function of the set $G$, which characterizes the prescribed control region $\Omega_{\text{sub}}$. The initial datum  $\phi$ is given in $L^{2}(\Omega)$, the positive coefficient $a_{0}\in L^{\infty}(Q)$ and $\nu$ is a positive constant. In addition, the control constraint is imposed through the admissible set $\mathcal{C}$ defined by
	\begin{equation}\label{eq:control constraint}
		\mathcal{C}=\left\{u|u\in L^{2}(G),~a \leq u(x,t) \leq b \text{~a.e.~} (x,t) \in G \right\},
	\end{equation}
	where $a$ and $b$ are given constants. The well-posedness of the optimal control problem \eqref{eq:obj} has been established under standard assumptions including the existence and uniqueness of an optimal solution; see, for instance, \cite{troeltzsch2010optimal, lions1971optimal, boyd2004convex}. Such problems arise in a wide range of applications, for example in clinical medicine \cite{glowinski2008exact} and engineering \cite{troeltzsch2010optimal}, among many others \cite{song2023numerical,hinze2009optimization,Wang2023Duality}. Figure~\ref{fig:multi_time}  illustrates an example of the evolution of the optimal control and its corresponding state.
	
	\begin{figure}[H]
		\centering
		
		\subfigure[$t=0.2$]{\includegraphics[width=0.23\linewidth]{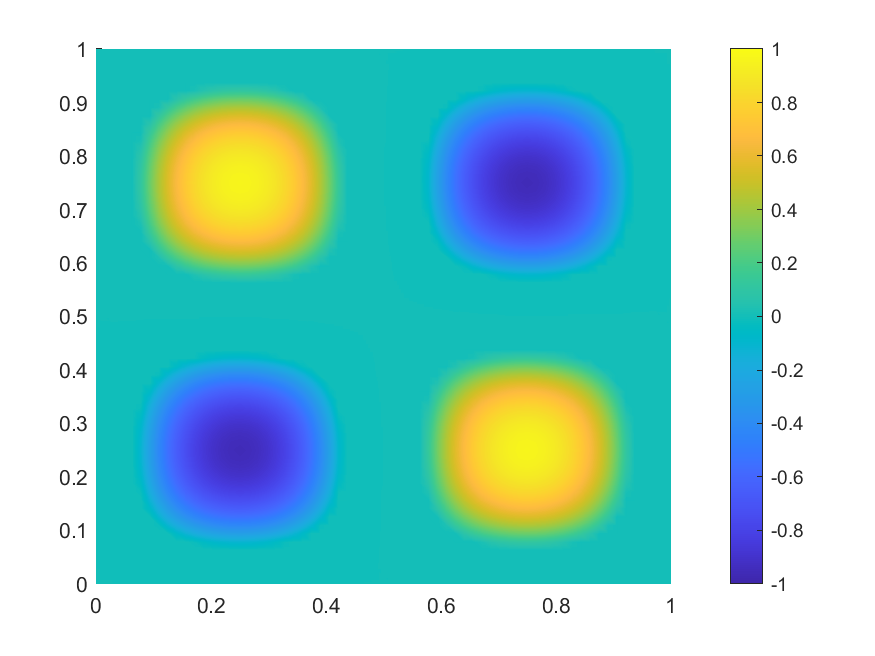}}\hfill
		\subfigure[$t=0.4$]{\includegraphics[width=0.23\linewidth]{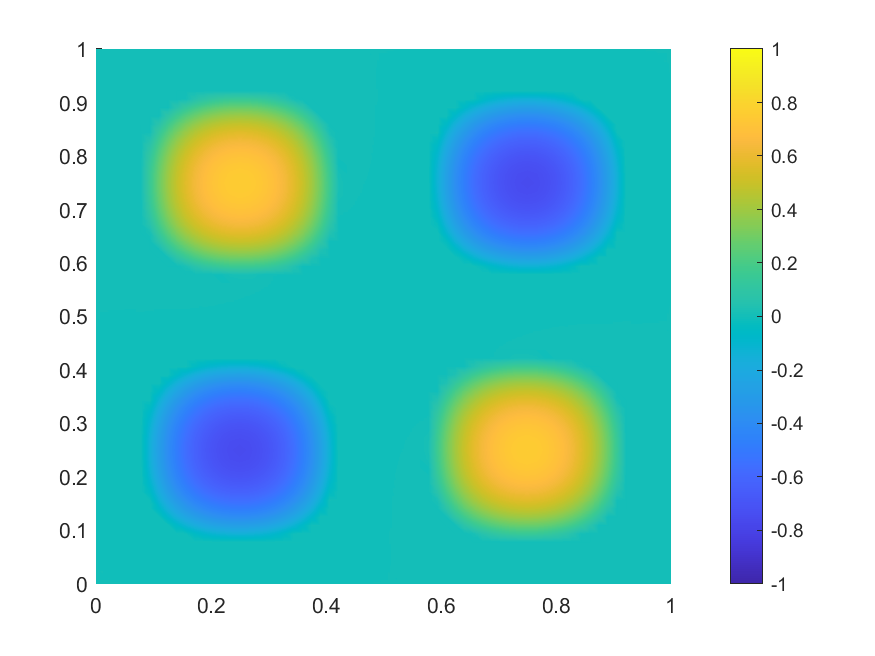}}\hfill
		\subfigure[$t=0.6$]{\includegraphics[width=0.23\linewidth]{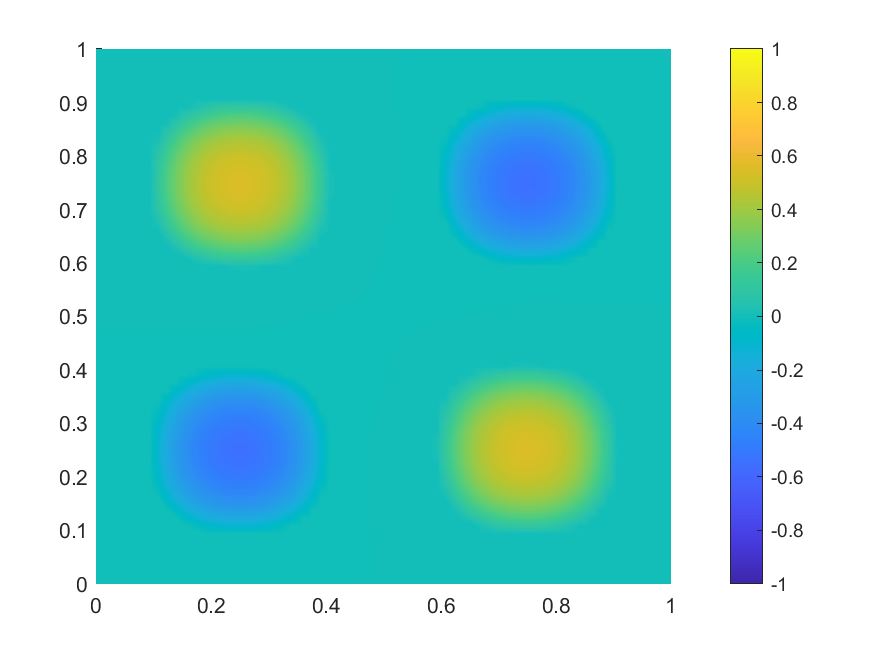}}\hfill
		\subfigure[$t=0.8$]{\includegraphics[width=0.23\linewidth]{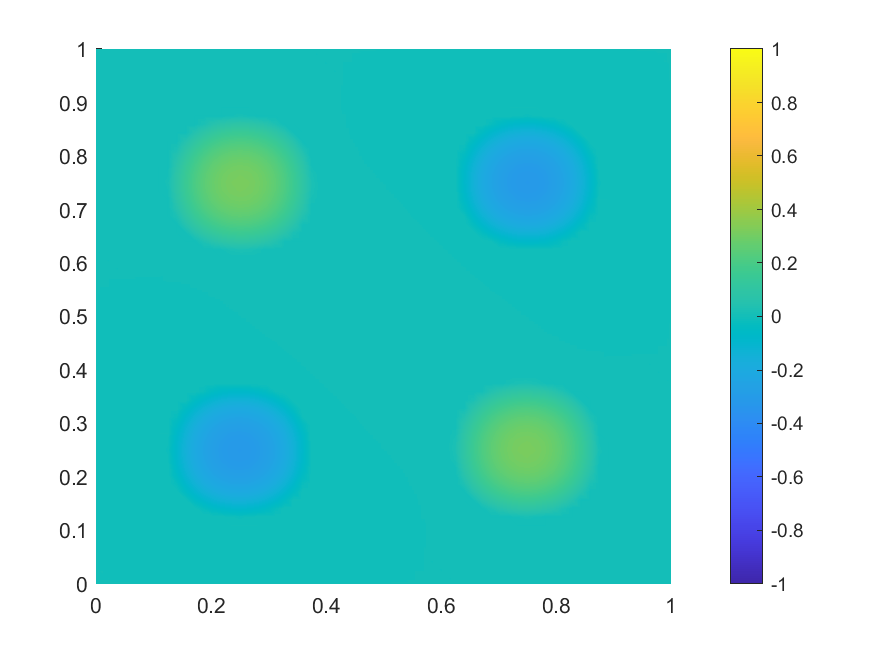}}
		
		\vspace{0.3cm}
		
		\subfigure[$t=0.2$]{\includegraphics[width=0.23\linewidth]{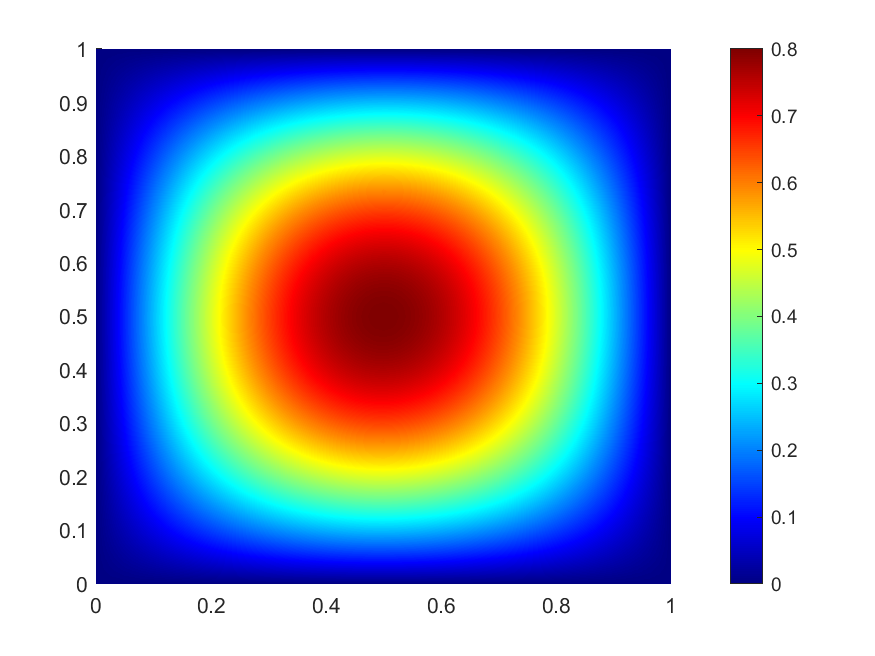}}\hfill
		\subfigure[$t=0.4$]{\includegraphics[width=0.23\linewidth]{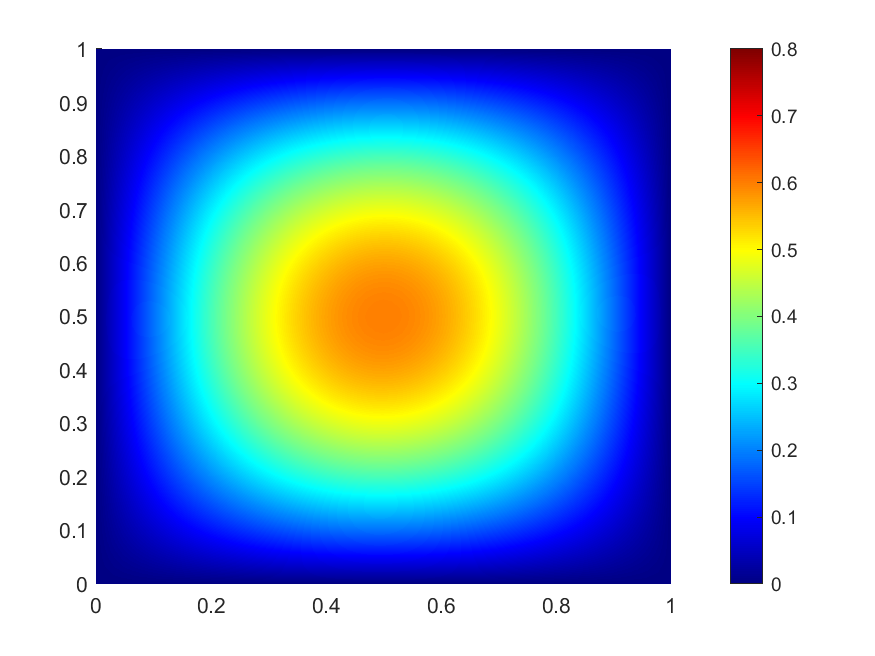}}\hfill
		\subfigure[$t=0.6$]{\includegraphics[width=0.23\linewidth]{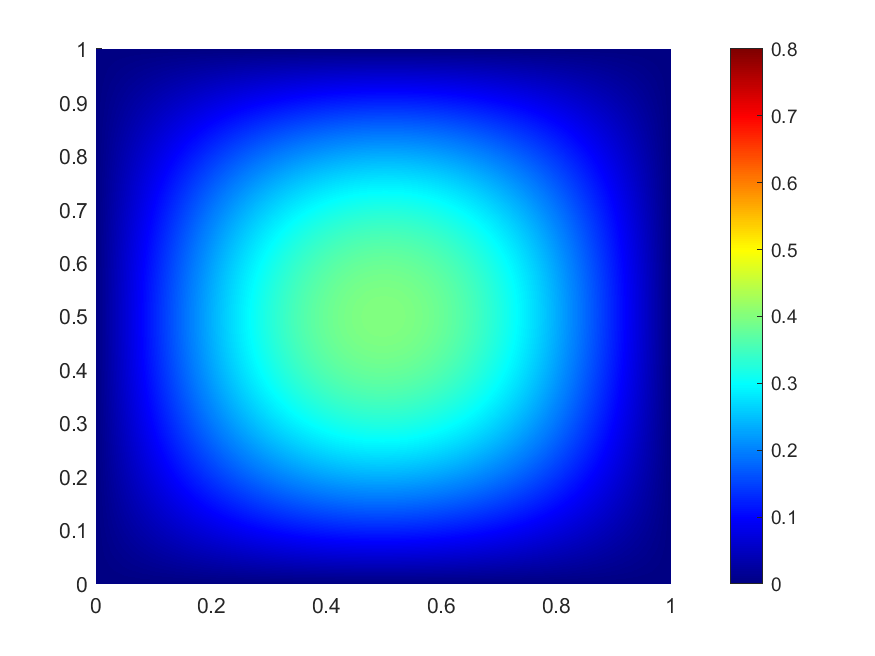}}\hfill
		\subfigure[$t=0.8$]{\includegraphics[width=0.23\linewidth]{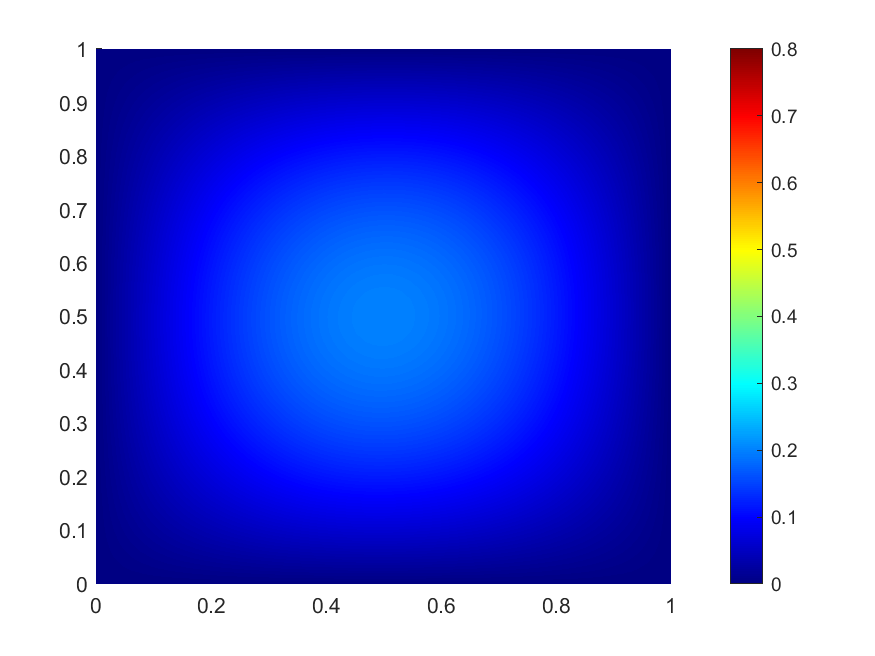}}
		
		\caption{Time-dependent control $u$ (top row) and the evolution of state $y$ (bottom row)}
		\label{fig:multi_time}
	\end{figure}
	
	In the special case that the control is unconstrained and the sparsity term is absent, i.e., $\mathcal{C}=L^{2}(G)$ and $\gamma_{s}=0$, the problem reduces to an unconstrained and smooth optimal control problem, which has been extensively studied theoretically and numerically in the literature, see, e.g., \cite{borzi2003multigrid,gander2016schwarz,glowinski2008exact,herzog2010algorithms,mathew2010analysis,mcdonald2016all,pearson2012regularization,rees2010optimal,ulbrich2007generalized}.
	Moreover, the projected gradient descent (PGD) methods, the semismooth Newton (SSN) methods \cite{casas2024convergence}, and the augmented Lagrangian methods (ALM) \cite{xu2021iteration} are popular approaches for constrained and nonsmooth optimal control problems.
	
	Since they are easy to implement, PGD methods remain among the methods of choice for complex problems. Note also that the parabolic optimal control problems require much more computational effort due to the additional time variable, and thus PGD methods are suitable for numerical simulations and constitute a valuable alternative to methods with higher orders of convergence; see, e.g., \cite{herzog2010algorithms,troeltzsch2010optimal,de2015numerical}. A crucial step in the implementation of PGD methods is the calculation of the descent stepsizes, which are usually expensive to calculate exactly or up to high accuracy. To estimate proper stepsizes, a common idea is to employ a line search strategy, such as a backtracking technique  based on the Armijo or Wolfe conditions; see, e.g., \cite{boyd2004convex,troeltzsch2010optimal,nocedal2006numerical}. It should be mentioned, however, that enforcing these line search conditions demands the repeated evaluation of the objective functional or its derivative, which may become very costly in practice because each evaluation involves solving the state equation and the corresponding adjoint equation repeatedly.
	
	The interest in SSN methods is driven by the increased complexity of the optimality conditions of the problem, due to the nonsmooth part of the objective and the indicator function associated with the control constraint, a situation in which PGD methods are not applicable; see, e.g., \cite{hinze2009optimization,troeltzsch2010optimal}. With a semismooth Newton direction constructed using the Clarke generalized Jacobian, SSN methods employ Newton-type iterations based on active-set strategies that iteratively identify active and inactive components with respect to the nonsmooth objective term and the control constraints. It was shown in \cite{hintermuller2002primal,hintermuller2010semismooth} that SSN methods exhibit local superlinear convergence for suitably chosen initial values, under the standard assumption that all Newton systems are solved exactly. Nevertheless, the Newton systems encountered in SSN methods are typically ill-conditioned, which underscores the need for appropriate preconditioners to ensure efficiency; see, e.g., \cite{stoll2014one,schiela2014operator,ulbrich2007generalized}. When employing SSN methods, the nonsmooth objective term and the control constraint are handled together with the associated parabolic equation; hence one has to recompute the preconditioner at each iteration as the active set varies. Furthermore, we often face unavoidable large-scale Newton systems when fine meshes are used to discretize time-dependent problems. Although the above issues can be mitigated by adaptive strategies \cite{bergounioux1999primal} and inner iterations \cite{porcelli2015preconditioning}, the computational load associated with solving the Newton systems still remains substantial; see, e.g., \cite{stoll2012preconditioning,stoll2014one}, and may impede the practical efficiency of SSN methods.
	
	As a fundamental approach to various constrained problems, the augmented Lagrangian method (ALM) has also been applied to optimal control problems. By handling the inequality constraints directly, the resulting iteration schemes become rather elaborate and complicated, which makes it challenging to analyze their convergence rate. The alternating direction method of multipliers (ADMM), which may be regarded as a splitting version of the classical ALM, has also been used to solve parabolic optimal control problems \cite{glowinski2022application,song2024admm}.
	
	Inspired by the above considerations, we aim to apply ADMM and to develop a uniform iteration scheme that can effectively solve the constrained parabolic optimal distributed control problem (\ref{eq:obj}) with a nonsmooth part of the objective and control constraints. The key advantage of ADMM lies in the fact that the main subproblem in ALM is decomposed and the resulting subproblems can be solved more easily at each iteration in a Gauss–Seidel fashion. Keeping this in mind, our first motivation is to separate the parabolic equation constraint (\ref{eq:state equation}) from the nonsmooth part of the objective in (\ref{eq:obj}) and the control constraints (\ref{eq:control constraint}) into two different subproblems. As a result, one of the subproblems should be tackled using iteration-independent coefficient matrices derived from the numerical discretization of the parabolic equation, whereas the other can be solved explicitly. However, it is important to note that the subproblem constrained by the parabolic equation is time-dependent and thus remains high-dimensional after space–time discretization. Although this subproblem can be solved by inner iterations, achieving high accuracy comes at a high computational cost. Hence, our second motivation is to design an adaptive inexact strategy for this subproblem, in order to alleviate the high cost of solving it exactly at each iteration. Numerically, the adaptive inexact strategy should be readily implementable, independent of the discretization mesh sizes, and capable of achieving sufficient accuracy. Theoretically, the adaptive inexact strategy must be robust enough to guarantee the convergence of the resulting inexact ADMM.
	
	The remainder of the paper is organized as follows. In Section~\ref{sec:Algorithmic framework of inexact ADMM}, we propose the algorithmic framework of the inexact ADMM based on a systematic inspection of the model problem. Section~\ref{sec:Convergence analysis} is devoted to the convergence analysis of the inexact ADMM framework. The numerical implementation and experiments are presented in Section~\ref{sec:numerical experiments}. Finally, conclusions are drawn in Section~\ref{sec:conclusion}.

	\section{Algorithmic framework}\label{sec:Algorithmic framework of inexact ADMM}
	In this section, we introduce some notation and elaborate on the essential concept of our inexact ADMM for the constrained parabolic distributed optimal control problem \eqref{eq:obj}.
	
	For clarity and convenience, we use $U$ and $Y$ for the spaces $L^{2}(G)$ and $L^{2}(Q)$, respectively.  We let $\langle\cdot,\cdot\rangle$ and $\|\cdot\|$ denote the $L^{2}$ inner product and norm with respect to $U$ or $Y$ without causing confusion. The $L^{1}$-norm on $U$ is denoted by $\| \cdot \|_{1}$. To simplify the presentation, we introduce the control-to-state operator $S:U\rightarrow Y$ defined by $y = S(u)$, associated with the state equation (\ref{eq:state equation}). By the superposition principle \cite{troeltzsch2010optimal}, it is evident that the affine operator $S$ can be represented as follows:
	\begin{equation*}\label{eq:affine operator}
		S(u) = \bar{S}u + S(0),
	\end{equation*}
	where the linear part $\bar{S} \in \mathcal{L}(U, Y)$ solves \eqref{eq:state equation} with zero initial data (i.e., $\phi=0$). We let $\bar{S}^{\ast}:Y\rightarrow U$ defines the adjoint operator of $\bar{S}$.

	By introducing an auxiliary variable $z$, the problem (\ref{eq:obj}) can be reformulated as:
	\begin{equation}\label{eq:reformulated model problem}
		\begin{aligned}
			\min\limits_{u,z\in U}&~ & &J(u)+R(z)\\
			\text{s.~t.}&~ & &u=z,
		\end{aligned}
	\end{equation}
	where
	\begin{equation*}
		\begin{aligned}
			J(u)=\frac{\gamma_{d}}{2}\|S(u)-y_{d}\|^{2}+\frac{1}{2}\|u\|^{2},\quad R(z)=\gamma_{s}\|z\|_{1}+I_{\mathcal{C}}(z).
		\end{aligned}
	\end{equation*}
	and $I_{\mathcal{C}}$ denotes the extended-valued indicator function of the admissible set $\mathcal{C}$ for the control constraints.
	With a penalty parameter $\beta>0$, the associated augmented Lagrangian functional is given by
	\begin{equation}\label{eq:augmented Lagrangian functional}
		L_{\beta}(u,z,\lambda):=J(u) +R(z)- \langle\lambda,u-z\rangle + \frac{\beta}{2}\| u-z\|^{2},
	\end{equation}
	where $\lambda\in  U $ is the Lagrange multiplier associated with the linear constraint $u=z$. Then, we can directly apply the classic ADMM to (\ref{eq:reformulated model problem}) as follows:
	\begin{subnumcases}
		{}
		\hat{u}^{k+1}= \mathop{\arg\min}\limits_{u\in U}
		L_{\beta}(u,\hat{z}^{k},\hat{\lambda}^{k}),\label{eq:classic ADMM-u}\\
		\hat{z}^{k+1}= \mathop{\arg\min}\limits_{z\in U}
		L_{\beta}(\hat{u}^{k+1},z,\hat{\lambda}^{k}),\label{eq:classic ADMM-z}\\
		\hat{\lambda}^{k+1}=\hat{\lambda}^{k}-\beta(\hat{u}^{k+1}-\hat{z}^{k+1}). \label{eq:classic ADMM-lambda}
	\end{subnumcases}
	Despite the concision of the above scheme, it is ideational that each subproblem requires exact minimization. In practice, not all subproblems can be easily solved in a high precision.
	
	Let us take a close look at the subproblems (\ref{eq:classic ADMM-u}) and (\ref{eq:classic ADMM-z}). From (\ref{eq:augmented Lagrangian functional}), it follows that the $u$-subproblem (\ref{eq:classic ADMM-u}) takes the form of the following smooth parabolic optimal control problem without control constraint:
	\begin{equation*}\label{eq:u-subrpoblem}
		\hat{u}^{k+1} = \underset{u\in U}{\arg\min} \left\{J(u) - \langle\hat{\lambda}^{k},u-\hat{z}^{k}\rangle + \frac{\beta}{2}\| u-\hat{z}^{k}\|^{2}\right\}.
	\end{equation*}
	With the constraint of state equation (\ref{eq:state equation}) presented by $y = S(u)$, the solution $\hat{u}^{k+1}$ satisfies
	\begin{equation}\label{eq:opt condition of u-subproblem}
		\hat{u}^{k+1} +\gamma_{d} \hat{p}^{k+1} - \hat{\lambda}^{k} + \beta (\hat{u}^{k+1}-\hat{z}^{k}) = 0,
	\end{equation}
	where the adjoint variable $\hat{p}^{k+1} = \bar{S}^{\ast}(S(\hat{u}^{k+1})-y_{d})$ satisfies the following adjoint equation:
	\begin{equation}\label{eq:adjoint equation k+1}
		\left\{
		\begin{aligned}
			&-\frac{\partial \hat{p}^{k+1}}{\partial t}-\nu\Delta \hat{p}^{k+1}+a_{0}\hat{p}^{k+1}=\hat{y}^{k+1}-y_{d},& \quad  &\text{\,in} \quad\Omega\times(0,T),\\
			&\hat{p}^{k+1}=0,& \quad  &\text{on}\quad\Gamma\times (0,T),\\
			&\hat{p}^{k+1}(T)=0,& \quad  &\text{\,in} \quad\Omega,
		\end{aligned}
		\right.
	\end{equation}
	in which $\hat{y}^{k+1}$ is the solution of
	\begin{equation}\label{eq:state equation k+1}
		\left\{
		\begin{aligned}
			&\frac{\partial \hat{y}^{k+1}}{\partial t}-\nu\Delta \hat{y}^{k+1}+a_{0}\hat{y}^{k+1}=\chi_{G} \hat{u}^{k+1},&\quad  &\text{\,in} \quad \Omega\times(0,T),\\
			&\hat{y}^{k+1}=0,&\quad  &\text{on}\quad\partial\Omega\times (0,T),\\
			&\hat{y}^{k+1}(0)=\phi& \quad  &\text{\,in} \quad\Omega.
		\end{aligned}
		\right.
	\end{equation}
	
	For the $z$-subproblem (\ref{eq:classic ADMM-z}), we have
	\begin{equation*}\label{eq:z-subproblem}
		\hat{z}^{k+1} = \underset{z\in U}{\arg\min} \left\{R(z) - \langle\hat{\lambda}^{k},\hat{u}^{k+1}-z\rangle + \frac{\beta}{2}\| \hat{u}^{k+1}-z\|^{2}\right\},
	\end{equation*}
	and it is easy to check that this subproblem admits the explicit solution
	\begin{equation*}\label{eq:z-subproblem solver}
		\hat{z}^{k+1} = P_{\mathcal{C}}\left(S_{\frac{\gamma_{s}}{\beta}}(\hat{u}^{k+1} - \beta^{-1}\hat{\lambda}^{k} )\right),
	\end{equation*}
	where the projection operator $P_{\mathcal{C}}$ and the soft-thresholding operator $ S_{\frac{\gamma_{s}}{\beta}}$ are defined pointwise by
	\begin{align*}
		&P_{\mathcal{C}}(v(x,t)) =\min \left\{b,\max\{v(x,t),a\}\right\},\\
		&S_{\frac{\gamma_{s}}{\beta}}(v(x,t))=\text {sign}(v(x,t))\max\{|v(x,t)|-\beta^{-1}\gamma_{s},0\}.
	\end{align*}
	
	Based on the above observations, it can be seen that the computational difficulties of solving the subproblems  (\ref{eq:classic ADMM-u}) and  (\ref{eq:classic ADMM-z}) are quite different. We can  obtain  the closed-form solution of $z$-subproblem  (\ref{eq:classic ADMM-z}) by the soft-thresholding operator and projection onto the set $\mathcal{C}$. Nevertheless, as a standard optimization problem constrained by the parabolic equation (\ref{eq:state equation}), the solution of $u$-subproblem  (\ref{eq:classic ADMM-u}) is essentially determined by the coupled time-dependent system (\ref{eq:opt condition of u-subproblem})-(\ref{eq:state equation k+1}) at $k$-th iteration, and consequently, attaining a solution to this system with high accuracy will spend expensive cost. Hence, it is advisable to solve the $u$-subproblem  (\ref{eq:classic ADMM-u}) iteratively and inexactly.
	
	A naive idea is to perform nested iterations such that
	\begin{equation}\label{eq:ideal inexact criterion}
		\|u^{k+1} - \tilde{u}^{k+1}\| \leq \theta_{k},
	\end{equation}
	where $u^{k+1}$ and $\tilde{u}^{k+1}$ denote, respectively, the inexact solution and the auxiliary solution of the $u$-subproblem associated with $(u^{k},z^{k},\lambda^{k})^{\top}$ at $k$-th iteration. Note that the auxiliary solution $\tilde{u}^{k+1}$ is distinct from the exact solution $\hat{u}^{k+1}$,  because here the solution $\tilde{u}^{k+1}$  depends on  the inexact triplet $(u^{k},z^{k},\lambda^{k})^{\top}$.
	
	Clearly, the above criterion can not be applied directly in practical implementations, and we should employ an executable criterion for the inexact iteration. To this end, we define $\sigma_{k}(u)$ by
	\begin{equation}\label{eq:criterion function}
		\sigma_{k}(u) = ( 1+\beta)u +\gamma_{d}\bar{S}^{\ast}(S(u)-y_{d}) - (\beta z^{k} + \lambda^{k}).
	\end{equation}
	From the optimality condition of $\tilde{u}^{k+1}$ and the definitions of operators $S$ and $\bar{S}^{\ast}$, it follows that $\sigma_{k}(\tilde{u}^{k+1}) = 0$ for any $k>0$. Given this, we propose the following criterion for solving $u$-subproblem inexactly:
	\begin{equation}\label{eq:inexact criterion}
		\|\sigma_{k}(u^{k+1})\|\leq \theta_{k},
	\end{equation}
	where $\theta_{k}$ is a criterion parameter. In addition, we let $\Theta_{k}$ denotes the set
	\begin{equation*}\label{eq:inexact solution set}
		\Theta_{k} = \{u|~\|\sigma_{k}(u)\|\leq \theta_{k} \}.
	\end{equation*}
	Then, for the problem (\ref{eq:obj}), the framework of inexact ADMM can be given by
	\begin{subnumcases}
		{}
		u^{k+1}\in \Theta_{k},\label{eq:inexact ADMM-u}\\
		z^{k+1}= P_{\mathcal{C}}\left(S_{\frac{\gamma_{s}}{\beta_{k}}}(u^{k+1} - \beta^{-1}_{k}\lambda^{k} )\right),\label{eq:inexact ADMM-z}\\
		\lambda^{k+1}=\lambda^{k}-\beta_{k}(u^{k+1}-z^{k+1}), \label{eq:inexact ADMM-lambda}
	\end{subnumcases}
	where the penalty parameter $\beta_{k}$ can be adjusted at each iteration.
	
	\begin{remark}
		It is  apparent that, in the above framework, the equation constraint have been decoupled from the nonsmooth  objective term and the control constraint at each iteration, and the subproblems (\ref{eq:inexact ADMM-z}) and (\ref{eq:inexact ADMM-lambda}) can be solved easily. For the subproblem (\ref{eq:inexact ADMM-u}), the criterion parameter $\theta_{k}$ does not need to be prescribed a priori as a fixed small constant and is independent of the discretization mesh sizes. Indeed, the criterion parameter $\theta_{k}$ can be set specified according to an appropriate rule (see Section~\ref{sec:numerical experiments} for details), and the inexact iterations can be performed automatically. In addition, the subproblem (\ref{eq:inexact ADMM-u}) can be solved by a suitable nested iteration, such as the conjugate gradient (CG) method described in Section~\ref{sec:numerical experiments}. With these features, the framework (\ref{eq:inexact ADMM-u})-(\ref{eq:inexact ADMM-lambda}) can be easily implemented.
	\end{remark}

	
	\section{Convergence analysis}\label{sec:Convergence analysis}
	In this section, we first present some mild assumptions and useful results for the subsequent analysis. Then we establish the convergence of the proposed inexact ADMM framework (\ref{eq:inexact ADMM-u})-(\ref{eq:inexact ADMM-lambda}).
	
	\subsection{Preliminaries}\label{sec:Preliminaries}
	We begin this subsection with the following assumptions that will be used throughout the paper.
	\begin{assumption}\label{assume:augmented parameter}
		The sequence of augmented penalty parameters $\{\beta_{k}\}_{0}^{\infty}$ satisfies
		\begin{equation*}
			\beta_{0}>0,\quad \frac{1}{1+\eta_{k}}\beta_{k} \leq \beta_{k+1}\leq (1+\eta_{k})\beta_{k},\quad \forall~ k\geq 0,
		\end{equation*}
		where $\{\eta_{k}\}_{0}^{\infty}$ is a non-negative sequence and $\sum_{k=0}^{\infty}\eta_{k} < +\infty$.
	\end{assumption}
	
	\begin{assumption}\label{assume:criterion parameter}
		The sequence of criterion parameters $\{\theta_{k}\}_{0}^{\infty}$ satisfies
		\begin{equation*}
			\sum_{k=0}^{\infty}\theta_{k} < +\infty .
		\end{equation*}
	\end{assumption}
	
	Note that under Assumption~\ref{assume:augmented parameter} the condition $ \sum_{k=0}^{\infty}\eta_{k} < +\infty$ implies that $\prod_{k=1}^{\infty}(1+\eta_{k})<+\infty$,  and hence that the sequence ${\beta_k}$ is uniformly bounded. More precisely, we have
	\begin{equation*}
		0<\underline{\beta}=\left( \prod_{k=1}^{\infty}(1+\eta_{k}) \right)^{-1}\beta_{0}\leq \beta_{k}\leq \left( \prod_{k=1}^{\infty}(1+\eta_{k}) \right) \beta_{0}=\overline{\beta}.
	\end{equation*}
	
	Next, we state some properties of the functional $J$ in \eqref{eq:reformulated model problem}.
	Let $DJ(u)$ denote the Fr\'{e}chet gradient of $J$ at $u$. Then  the gradient operator $DJ$ is Lipschitz continuous and there exists a constant $L = 1 + \gamma_d \|\bar{S}^*\bar{S}\|$ such that
	\begin{equation}\label{eq:L-Lip}
		\|DJ(u_{1})-DJ(u_{2})\|\leq L\|u_{1}-u_{2}\|,\quad \forall~u_{1},u_{2}\in U.
	\end{equation}
	As a standard result (see, e.g., \cite{nesterov2004introductory}),  the following inequality holds:
	\begin{equation}\label{eq:descent of Lip}
		J(u_{1})\geq J(u_{2})+\langle DJ(u_{2}),u_{1}-u_{2}\rangle+\dfrac{1}{2L}\|DJ(u_{1})-DJ(u_{2})\|^{2}, \quad \forall~u_{1},u_{2}\in U.
	\end{equation}
	Moreover, the functional $J$ is also strongly convex and the gradient operator $DJ$ is strongly monotone, that is,
	\begin{equation}\label{eq:strongly monotone}
		\langle DJ(u_{1})-DJ(u_{2}), u_{1}-u_{2}\rangle \geq  \|u_{1}-u_{2}\|^{2},
		\quad\forall~ u_{1},u_{2}\in U,
	\end{equation}
	with modulus $1$. Consequently, it holds that
	\begin{equation}\label{eq:desent of strong con}
		J(u_{1})\geq J(u_{2})+\langle DJ(u_{2}),u_{1}-u_{2}\rangle+\dfrac{1}{2}\|u_{1}-u_{2}\|^{2}, \quad \forall~u_{1},u_{2}\in U.
	\end{equation}

	To simplify the presentation, we use $\omega$ for $(u,z,\lambda)^{\top}$, $\upsilon$ for $(z,\lambda)^{\top}$, and denote $W$ and $V$ as $U\times \mathcal{C} \times \Lambda$ and $\mathcal{C}\times \Lambda$, respectively. For any $v \in V$, we define $M_{k}$-induced norm associated with the matrix
	\begin{equation*}\label{Matrix:theo3}
		M_{k} =
		\begin{pmatrix}
			\beta_{k}I &0\\ 0& \beta_{k}^{-1}I
		\end{pmatrix},
	\end{equation*}
	by
	\begin{equation}\label{eq:induced norm}
		\|v\|_{M_{k}} = \langle v,M_{k} v \rangle^{\frac{1}{2}} = \left( \beta_{k}\|z\|^{2} + \beta_{k}^{-1}\|\lambda\|^{2}\right)^{\frac{1}{2}} .
	\end{equation}

	By the definition of the augmented Lagrangian functional (\ref{eq:augmented Lagrangian functional}), a point $\omega^{\ast} = (u^{\ast},z^{\ast},\lambda^{\ast})^{\top}\in W$ is an optimal solution to problem (\ref{eq:reformulated model problem}) if and only if it satisfies the following variational inequality \cite{lin2022alternating}:
	\begin{equation}\label{eq:variational inequality}
		J(u)+R(z)-J(u^{\ast})-R(z^{\ast})-\langle \lambda^{\ast} ,u-z\rangle \geq 0,
		\quad \forall~\omega\in W.
	\end{equation}
	Recall that in the classic ADMM \eqref{eq:classic ADMM-u}-\eqref{eq:classic ADMM-lambda}, all subproblems are assumed to be solved exactly. Hence, we obtain from the optimality conditions for $\hat{\omega}^{k+1} = (\hat{u}^{k+1},\hat{z}^{k+1},\hat{\lambda}^{k+1})^{\top}$ that
	\begin{subnumcases}
		{}
		DJ(\hat{u}^{k+1})-(\hat{\lambda}^{k}-\beta_{k}(\hat{u}^{k+1}-\hat{z}^{k})) =0,
		\label{eq: OptCondition of Classic ADMM-u}\\[1mm]
		R(z)-R(\hat{z}^{k+1})+\left\langle z-\hat{z}^{k+1},\hat{\lambda}^{k}-\beta_{k}(\hat{u}^{k+1}-\hat{z}^{k+1})\right\rangle\geq 0,
		\quad \forall~ z\in U,
		\label{eq: OptCondition of Classic ADMM-z}\\[1mm]
		\hat{\lambda}^{k+1}=\hat{\lambda}^{k}-\beta_{k}(\hat{u}^{k+1}-\hat{z}^{k+1}).\label{eq: OptCondition of Classic ADMM-lambda}
	\end{subnumcases}	
	We denote by $\tilde{\omega}^{k+1} = (\tilde{u}^{k+1}, \tilde{z}^{k+1}, \tilde{\lambda}^{k+1})^{\top}$ the auxiliary solution, which is the exact solution obtained at $k$-th by updating the inexact solution $\omega^{k}$ according to the classic ADMM  (\ref{eq: OptCondition of Classic ADMM-u})-(\ref{eq: OptCondition of Classic ADMM-lambda}). Clearly, for the auxiliary ``exact solution''  $\tilde{\omega}^{k+1}$, the optimality conditions are given by
	\begin{subnumcases}
		{}
		DJ(\tilde{u}^{k+1})-(\lambda^{k}-\beta_{k}(\tilde{u}^{k+1}-z^{k})) =0,
		\label{eq: OptCondition of auxiliary ADMM-u}\\[1mm]
		R(z)-R(\tilde{z}^{k+1})+\left\langle z-\tilde{z}^{k+1},\tilde{\lambda}^{k}-\beta_{k}(\tilde{u}^{k+1}-\tilde{z}^{k+1})\right\rangle\geq 0, \quad \forall~ z\in U,
		\label{eq: OptCondition of auxiliary ADMM-z}\\[1mm]
		\tilde{\lambda}^{k+1}=\lambda^{k}-\beta_{k}(\tilde{u}^{k+1}-\tilde{z}^{k+1}).
		\label{eq: OptCondition of auxiliary ADMM-lambda}
	\end{subnumcases}
	In addition, we have the following relation between the auxiliary solution and the inexact solution.
	
	\begin{proposition}\label{lem:inexact criterion}
		Let $\omega^{k}$ be the sequence generated by the inexact ADMM framework \eqref{eq:inexact ADMM-u}-\eqref{eq:inexact ADMM-lambda}, and $\tilde{\omega}^{k+1}$ be the corresponding auxiliary solution satisfying \eqref{eq: OptCondition of auxiliary ADMM-u}-\eqref{eq: OptCondition of auxiliary ADMM-lambda}. Then, for all $k \geq 0$, we have
		\begin{equation*}\label{eq:inexact criterion 0}
			\| u^{k+1}-\tilde{u}^{k+1}\| \leq \theta_{k}.
		\end{equation*}
	\end{proposition}
	
	\begin{proof}
		From the definition \eqref{eq:criterion function} and the positivity of $\beta_{k}$, it is easy to check that $\sigma_{k}$ is strongly monotone with modulus $\mu_{k}\geq 1$.
		
		For any $u^{k+1},~\tilde{u}^{k+1}\in U$, using the identity
		\begin{equation}\label{eq:cos}
			\begin{aligned}
				\left\langle b-a, a-c\right\rangle = \frac{1}{2}(\| b-c\|^{2}-\|a-c\|^{2}-\|b-a\|^{2}),
			\end{aligned}
		\end{equation}
		it holds that
		\begin{equation*}
			\begin{aligned}
				2\langle u^{k+1}-\tilde{u}^{k+1},\sigma_{k}(u^{k+1})\rangle
				=\| u^{k+1}-\tilde{u}^{k+1}\|^{2}+\| \sigma_{k}(u^{k+1})\|^{2}
				-\| u^{k+1}-\sigma_{k}(u^{k+1})-\tilde{u}^{k+1}\|^{2}.
			\end{aligned}
		\end{equation*}
		Based on the above equation, the fact that $\sigma_{k}(\tilde{u}^{k+1})=0$, and the strongly monotone property of $\sigma_{k}$, we have
		\begin{equation}\label{eq:inexact criterion 1}
			\begin{aligned}
				\| \sigma_{k}(u^{k+1})\|^{2}
				\geq& 2\langle u^{k+1}-\tilde{u}^{k+1},\sigma_{k}(u^{k+1})\rangle-\| u^{k+1}-\tilde{u}^{k+1}\|^{2}\\
				=&2\langle u^{k+1}-\tilde{u}^{k+1},\sigma_{k}(u^{k+1})-\sigma_{k}(\tilde{u}^{k+1})\rangle-\| u^{k+1}-\tilde{u}^{k+1}\|^{2}\\
				\geq&(2\mu_{k}-1)\| u^{k+1}-\tilde{u}^{k+1}\|^{2}.
			\end{aligned}
		\end{equation}
		Since the modulus $\mu_{k}\geq 1$, applying (\ref{eq:inexact criterion}), the proof is completed.
	\end{proof}
	It is important to emphasize that the auxiliary solution ${\tilde{\omega}^{k}}$ is introduced solely for convergence analysis and is not required in practical computations. While Proposition \ref{lem:inexact criterion}  provides a bound on the inexactness of $u^{k}$, a completed  convergence proof for the inexact ADMM (\ref{eq:inexact ADMM-u})-(\ref{eq:inexact ADMM-lambda}) requires  additional properties of the exact ADMM sequence ${\hat{\omega}^{k}}$ generated by (\ref{eq: OptCondition of Classic ADMM-u})-(\ref{eq: OptCondition of Classic ADMM-lambda}).

	\subsection{Properties of exact framework (\ref{eq: OptCondition of Classic ADMM-u})-(\ref{eq: OptCondition of Classic ADMM-lambda})}\label{sec:Properties of exact framework}
	
	This subsection presents some crucial properties of the exact ADMM iteration (\ref{eq: OptCondition of Classic ADMM-u})–(\ref{eq: OptCondition of Classic ADMM-lambda}), which will be used to derive the convergence of the proposed inexact framework.

	\begin{lemma}\label{lem:estimate for KKT}
		Let $\{\hat{\omega}^{k}\}$  be the sequence generated by (\ref{eq: OptCondition of Classic ADMM-u})-(\ref{eq: OptCondition of Classic ADMM-lambda}). Then for any $k\geq 0$, it holds that
		\begin{equation}\label{eq:estimate for KKT}
			\begin{aligned}
				& \left \langle DJ(\hat{u}^{k+1}),\hat{u}^{k+1}-u^{\ast}\right\rangle+\left\langle \partial R(\hat{z}^{k+1}),\hat{z}^{k+1}-z^{\ast} \right \rangle -\left\langle \lambda^{\ast},\hat{u}^{k+1}-\hat{z}^{k+1}\right\rangle\\
				\leq & \,\frac{1}{2}(\| \hat{v}^{k}-v^{\ast}\|_{M_{k}}^{2}-\| \hat{v}^{k+1}-v^{\ast}\|_{M_{k}}^{2}-\| \hat{v}^{k+1}-\hat{v}^{k}\|_{M_{k}}^{2}),
			\end{aligned}
		\end{equation}
		where $\partial R$ is the subdifferential of the functional $R$.
	\end{lemma}
	\begin{proof}
		Substituting \eqref{eq: OptCondition of Classic ADMM-lambda} into \eqref{eq: OptCondition of Classic ADMM-u} and the optimality condition of \eqref{eq: OptCondition of Classic ADMM-z}, respectively, we obtain
		\begin{equation}\label{eq:desent of exact 1}
			DJ(\hat{u}^{k+1})=\hat{\lambda}^{k+1}-\beta_{k}(\hat{z}^{k+1}-\hat{z}^{k})\quad \mbox{and} \quad -\hat{\lambda}^{k+1} \in \partial R(\hat{z}^{k+1}).
		\end{equation}
		Further, using the relations (\ref{eq:desent of exact 1}), the fact $u^{\ast} = z^{\ast}$ and the equation (\ref{eq: OptCondition of Classic ADMM-lambda}) in order,
		we derive
		\begin{equation}\label{eq:desent of exact 2}
			\begin{aligned}
				&\left\langle DJ(\hat{u}^{k+1}),\hat{u}^{k+1}-u^{\ast}\rangle+ \langle \partial R(\hat{z}^{k+1}),\hat{z}^{k+1}-z^{\ast} \right \rangle
				-\langle\lambda^{\ast},\hat{u}^{k+1}-\hat{z}^{k+1}\rangle \\
				=
				&\langle\hat{\lambda}^{k+1},\hat{u}^{k+1}-\hat{z}^{k+1}\rangle
				-\beta_{k}\langle\hat{z}^{k+1}-\hat{z}^{k},\hat{u}^{k+1}-\hat{z}^{k+1}+\hat{z}^{k+1}-z^{\ast}\rangle
				-\langle\lambda^{\ast},\hat{u}^{k+1}-\hat{z}^{k+1}\rangle\\
				=
				&\beta_{k}^{-1}\langle\lambda^{\ast}-\hat{\lambda}^{k+1},\hat{\lambda}^{k+1}-\hat{\lambda}^{k}\rangle
				+\langle\hat{z}^{k+1}-\hat{z}^{k},\hat{\lambda}^{k+1}-\hat{\lambda}^{k}\rangle+\beta_{k}\langle \hat{z}^{k}-\hat{z}^{k+1},\hat{z}^{k+1}-z^{\ast}\rangle.
			\end{aligned}
		\end{equation}
		
		By the monotonicity of the subdifferential of a convex functional, the second term on the right-hand side of \eqref{eq:desent of exact 2} is nonpositive. Applying the identity \eqref{eq:cos} to the remaining two inner products in \eqref{eq:desent of exact 2}, we obtain the inequality
		\begin{equation}\label{eq:desent of auxiliary and inexact 4}
			\begin{aligned}
				& \langle DJ(\hat{u}^{k+1}),\hat{u}^{k+1}-u^{\ast}\rangle+\left \langle \partial R(\hat{z}^{k+1}),\hat{z}^{k+1}-z^{\ast} \right \rangle
				-\langle\lambda^{\ast},\hat{u}^{k+1}-\hat{z}^{k+1}\rangle \\
				\leq
				& \frac{\beta_{k}^{-1}}{2} \| \lambda^{\ast}-\hat{\lambda}^{k}\|^{2}-\frac{\beta_{k}^{-1}}{2} \| \lambda^{\ast}-\hat{\lambda}^{k+1}\|^{2}-\frac{\beta_{k}^{-1}}{2} \| \hat{\lambda}^{k+1}-\hat{\lambda}^{k}\|^{2}\\
				&+\frac{\beta_{k}}{2} \| z^{\ast}-\hat{z}^{k}\|^{2}-\frac{\beta_{k}}{2} \| z^{\ast}-\hat{z}^{k+1}\|^{2}-\frac{\beta_{k}}{2} \| \hat{z}^{k+1}-\hat{z}^{k}\|^{2}\\
				=
				&\frac{1}{2}(\| \hat{v}^{k}-v^{\ast}\|_{M_{k}}^{2}-\| \hat{v}^{k+1}-v^{\ast}\|_{M_{k}}^{2}-\| \hat{v}^{k+1}-\hat{v}^{k}\|_{M_{k}}^{2}).
			\end{aligned}
		\end{equation}
		This completes the proof of Lemma \ref{lem:estimate for KKT}.
	\end{proof}

	\begin{lemma}\label{lem: exact KKT est}
		Let $\{\hat{\omega}^{k}\}$ be the sequence generated by (\ref{eq: OptCondition of Classic ADMM-u})-(\ref{eq: OptCondition of Classic ADMM-lambda}). Then for any $k\geq 0$, it holds that
		\begin{equation}\label{eq: desent of KKT strong con }
			\begin{aligned}
				&J(\hat{u}^{k+1})+R(\hat{z}^{k+1})-J(u^{\ast})-R(z^{\ast})-\langle \lambda^{\ast} ,\hat{u}^{k+1}-\hat{z}^{k+1}\rangle\\
				\leq
				&\frac{1}{2}(\| \hat{v}^{k}-v^{\ast}\|_{M_{k}}^{2}-\| \hat{v}^{k+1}-v^{\ast}\|_{M_{k}}^{2}-\| \hat{v}^{k+1}-\hat{v}^{k}\|_{M_{k}}^{2}-\|\hat{u}^{k+1}-u^{\ast}\|^{2})
			\end{aligned}
		\end{equation}
		and
		\begin{equation}\label{eq: desent of KKT L-Lip}
			\begin{aligned}
				&J(\hat{u}^{k+1})+R(\hat{z}^{k+1})-J(u^{\ast})-R(z^{\ast})-\langle \lambda^{\ast} ,\hat{u}^{k+1}-\hat{z}^{k+1}\rangle\\
				\leq
				&\frac{1}{2}(\| \hat{v}^{k}-v^{\ast}\|_{M_{k}}^{2}-\| \hat{v}^{k+1}-v^{\ast}\|_{M_{k}}^{2}-\| \hat{v}^{k+1}-\hat{v}^{k}\|_{M_{k}}^{2}-\dfrac{1}{L}\|DJ(\hat{u}^{k+1})-DJ(u^{\ast})\|^{2}).
			\end{aligned}
		\end{equation}				
	\end{lemma}
	\begin{proof}
		Combining the inequality (\ref{eq:desent of strong con}), the convexity of $ R(z)$ and the inequality (\ref{eq:estimate for KKT}), we have
		\begin{equation*}
			\begin{aligned}
				&J(\hat{u}^{k+1})+R(\hat{z}^{k+1})-J(u^{\ast})-R(z^{\ast})-\langle \lambda^{\ast} ,\hat{u}^{k+1}-\hat{z}^{k+1}\rangle\\
				\leq
				&\left \langle DJ(\hat{u}^{k+1}),\hat{u}^{k+1}-u^{\ast}\right\rangle+\left \langle \partial R(\hat{z}^{k+1}),\hat{z}^{k+1}-z^{\ast} \right \rangle -\left\langle \lambda^{\ast},\hat{u}^{k+1}-\hat{z}^{k+1}\right\rangle-\dfrac{1}{2}\| \hat{u}^{k+1}-u^{\ast}\|^{2}\\
				\leq
				&\frac{1}{2}(\| \hat{v}^{k}-v^{\ast}\|_{M_{k}}^{2}-\| \hat{v}^{k+1}-v^{\ast}\|_{M_{k}}^{2}-\| \hat{v}^{k+1}-\hat{v}^{k}\|_{M_{k}}^{2}-\|\hat{u}^{k+1}-u^{\ast}\|^{2}).
			\end{aligned}
		\end{equation*}
		Similarly, we can obtain (\ref{eq: desent of KKT L-Lip}) from (\ref{eq:descent of Lip}).
	\end{proof}
	
	\begin{theorem}\label{thm: exact Convergence rate}
		Let $\{\hat{\omega}^{k}\}$  be the sequence generated by (\ref{eq: OptCondition of Classic ADMM-u})-(\ref{eq: OptCondition of Classic ADMM-lambda}). Under Assumption 1, there exists $  0<q_{1}<1$, such that
		\begin{equation*}
			\|\hat{v}^{k+1}-v^{\ast}\|_{M_{k}}\leq q_{1} \|\hat{v}^{k}-v^{\ast}\|_{M_{k}}.
		\end{equation*}
	\end{theorem}
	
	\begin{proof}
		First, from equation (\ref{eq: OptCondition of Classic ADMM-u}) and the optimality of $\omega^{\ast}$, we derive
		\begin{equation}\label{eq:exact Convergence rate 1}
			\begin{aligned}
				&\|DJ(\hat{u}^{k+1})-DJ(u^{\ast})\|^{2}\\
				=&\,\|\hat{\lambda}^{k+1}-\lambda^{\ast}+\beta_{k}(\hat{z}^{k}-\hat{z}^{k+1})\|^{2}\\
				=&\, \|\hat{\lambda}^{k+1}-\lambda^{\ast}\|^{2}+\beta_{k}^{2}\|\hat{z}^{k}-\hat{z}^{k+1}\|^{2}+2\beta_{k}\langle \hat{\lambda}^{k+1}-\lambda^{\ast},\hat{z}^{k}-\hat{z}^{k+1}\rangle\\
				\geq &\, \|\hat{\lambda}^{k+1}-\lambda^{\ast}\|^{2}+\beta_{k}^{2}\|\hat{z}^{k}-\hat{z}^{k+1}\|^{2}-\epsilon_{1}\|\hat{\lambda}^{k+1}-\lambda^{\ast}\|^{2}-\beta_{k}^2{\epsilon_{1}^{-1}}\|\hat{z}^{k}-\hat{z}^{k+1}\|^{2}\\
				=&\,(1-\epsilon_{1})\|\hat{\lambda}^{k+1}-\lambda^{\ast}\|^{2}+\beta_{k}^{2}(1-{\epsilon_{1}^{-1}})\|\hat{z}^{k}-\hat{z}^{k+1}\|^{2},
			\end{aligned}
		\end{equation}
		where the last line follows from Young's inequality with a parameter $1/(1+L/\bar{\beta})<\epsilon_{1}<1$.
		
		Next, from the multiplier update (\ref{eq: OptCondition of Classic ADMM-lambda}), we obtain a similar estimate for the primal variable:
		\begin{equation}\label{eq:exact Convergence rate 2}
			\begin{aligned}
				&\|\hat{u}^{k+1}-u^{\ast}\|^{2}\\
				=&\,\|\hat{u}^{k+1}-\hat{z}^{k+1}+\hat{z}^{k+1}-z^{\ast}\|^{2}\\
				\geq&\, (1-\epsilon_{2})\|\hat{z}^{k+1}-z^{\ast}\|^{2}+(1-{\epsilon_{2}^{-1}})\|\hat{u}^{k+1}-\hat{z}^{k+1}\|^{2}\\
				=&\,(1-\epsilon_{2})\|\hat{z}^{k+1}-z^{\ast}\|^{2}+{\beta_{k}^{-2}}(1-{\epsilon_{2}^{-1}})\|\hat{\lambda}^{k+1}-\hat{\lambda}^{k}\|^{2},
			\end{aligned}
		\end{equation}
		where $1/(1+\underline{\beta})<\epsilon_{2}<1$.
		
		We now combine the above estimates with the inequalities from Lemma \ref{lem: exact KKT est}. Substituting (\ref{eq:exact Convergence rate 1}) into (\ref{eq: desent of KKT L-Lip}) gives
		\begin{align}
			&J(\hat{u}^{k+1})+R(\hat{z}^{k+1})-J(u^{\ast})-R(z^{\ast})-\langle \lambda^{\ast} ,\hat{u}^{k+1}-\hat{z}^{k+1}\rangle \nonumber \\
			\leq
			&\frac{1}{2}(\| \hat{v}^{k}-v^{\ast}\|_{M_{k}}^{2}-\| \hat{v}^{k+1}-v^{\ast}\|_{M_{k}}^{2})\nonumber \\
			~~&-\frac{1}{2}\| \hat{v}^{k+1}-\hat{v}^{k}\|_{M_{k}}^{2}-\dfrac{(1-\epsilon_{1})}{2L}\|\hat{\lambda}^{k+1}-\lambda^{\ast}\|^{2}-\dfrac{\beta_{k}^{2}(1-{\epsilon_{1}^{-1}})}{2L}\|\hat{z}^{k}-\hat{z}^{k+1}\|^{2} \label{eq:exact Convergence rate 3} \\
			\leq
			&\frac{1}{2}(\| \hat{v}^{k}-v^{\ast}\|_{M_{k}}^{2}-\| \hat{v}^{k+1}-v^{\ast}\|_{M_{k}}^{2}) \nonumber\\
			~~&-\frac{\beta_{k}^{-1}}{2}\| \hat{\lambda}^{k+1}-\hat{\lambda}^{k}\|^{2}-\dfrac{(1-\epsilon_{1})}{2L}\|\hat{\lambda}^{k+1}-\lambda^{\ast}\|^{2}+\left(\dfrac{\beta_{k}^{2}({\epsilon_{1}^{-1}}-1)}{2L}-\dfrac{\beta_{k}}{2}\right)\|\hat{z}^{k}-\hat{z}^{k+1}\|^{2} \nonumber\\
			\leq
			&\frac{1}{2}(\| \hat{v}^{k}-v^{\ast}\|_{M_{k}}^{2}-\| \hat{v}^{k+1}-v^{\ast}\|_{M_{k}}^{2})-\dfrac{(1-\epsilon_{1})}{2L}\|\hat{\lambda}^{k+1}-\lambda^{\ast}\|^{2}. \nonumber
		\end{align}
		Similarly, substituting (\ref{eq:exact Convergence rate 2}) into (\ref{eq: desent of KKT strong con }) yields
		\begin{equation}\label{eq:exact Convergence rate 4}
			\begin{aligned}
				&J(\hat{u}^{k+1})+R(\hat{z}^{k+1})-J(u^{\ast})-R(z^{\ast})-\langle \lambda^{\ast} ,\hat{u}^{k+1}-\hat{z}^{k+1}\rangle\\
				\leq
				&\frac{1}{2}(\| \hat{v}^{k}-v^{\ast}\|_{M_{k}}^{2}-\| \hat{v}^{k+1}-v^{\ast}\|_{M_{k}}^{2})\\
				~~&-\frac{\beta_{k}}{2}\| \hat{z}^{k+1}-\hat{z}^{k}\|^{2}-\dfrac{(1-\epsilon_{2})}{2}\|\hat{z}^{k+1}-z^{\ast}\|^{2}+\left(\dfrac{({\epsilon_{2}^{-1}}-1)\beta_{k}^{-2}}{2}-\dfrac{\beta_{k}^{-1}}{2}\right)\|\hat{\lambda}^{k}-\hat{\lambda}^{k+1}\|^{2}\\
				\leq
				&\frac{1}{2}(\| \hat{v}^{k}-v^{\ast}\|_{M_{k}}^{2}-\| \hat{v}^{k+1}-v^{\ast}\|_{M_{k}}^{2})-\dfrac{(1-\epsilon_{2})}{2}\|\hat{z}^{k+1}-z^{\ast}\|^{2}.
			\end{aligned}
		\end{equation}
		
		Since the left-hand sides of both (\ref{eq:exact Convergence rate 3}) and (\ref{eq:exact Convergence rate 4}) are nonnegative by the inequality (\ref{eq:variational inequality}), we obtain two key inequalities:
		\begin{align}
			&\dfrac{(1-\epsilon_{1})}{2L}\|\hat{\lambda}^{k+1}-\lambda^{\ast}\|^{2} \nonumber\\
			\leq &\, \dfrac{\beta_{k}^{-1}}{2}\|\hat{\lambda}^{k}-\lambda^{\ast}\|^{2}-\dfrac{\beta_{k}^{-1}}{2}\|\hat{\lambda}^{k+1}-\lambda^{\ast}\|^{2}
			+\dfrac{\beta_{k}}{2}\|\hat{z}^{k}-z^{\ast}\|^{2}-\dfrac{\beta_{k}}{2}\|\hat{z}^{k+1}-z^{\ast}\|^{2}, \label{eq:exact Convergence rate 5}\\
			&\dfrac{(1-\epsilon_{2})}{2}\|\hat{z}^{k+1}-z^{\ast}\|^{2} \nonumber \\
			\leq &\, \dfrac{\beta_{k}^{-1}}{2}\|\hat{\lambda}^{k}-\lambda^{\ast}\|^{2}-\dfrac{\beta_{k}^{-1}}{2}\|\hat{\lambda}^{k+1}-\lambda^{\ast}\|^{2}
			+\dfrac{\beta_{k}}{2}\|\hat{z}^{k}-z^{\ast}\|^{2}-\dfrac{\beta_{k}}{2}\|\hat{z}^{k+1}-z^{\ast}\|^{2}.  \label{eq:exact Convergence rate 6}
		\end{align}
		To combine these estimates into a single contraction for the weighted norm $\|\cdot\|_{M_k}$, we take a convex combination of (\ref{eq:exact Convergence rate 5}) and (\ref{eq:exact Convergence rate 6}). Specifically, multiply (\ref{eq:exact Convergence rate 5}) by $t_k$ and (\ref{eq:exact Convergence rate 6}) by $(1-t_k)$, then add them to obtain
		\begin{align*}
			&\left(\dfrac{(1-\epsilon_{1})t_k}{2L}+\dfrac{\beta_{k}^{-1}}{2}\right)\|\hat{\lambda}^{k+1}-\lambda^{\ast}\|^{2}+\left(\dfrac{(1-\epsilon_{2})(1-t_k)}{2}+\dfrac{\beta_{k}}{2}\right)\|\hat{z}^{k+1}-z^{\ast}\|^{2}\\
			\leq
			&\,\dfrac{\beta_{k}^{-1}}{2}\|\hat{\lambda}^{k}-\lambda^{\ast}\|^{2}+\dfrac{\beta_{k}}{2}\|\hat{z}^{k}-z^{\ast}\|^{2}.
		\end{align*}
		For this inequality to imply a contraction $\|\hat{v}^{k+1}-v^{\ast}\|_{M_{k}}\leq q_{1,k} \|\hat{v}^{k}-v^{\ast}\|_{M_{k}}$, we require the coefficients on the left to be exactly $1/(2\beta_k q_{1,k}^2)$ and $\beta_k/(2 q_{1,k}^2)$ respectively. This condition gives the system:
		\begin{equation*}
			\left(\dfrac{(1-\epsilon_{1})t_k}{2L}+\dfrac{\beta_{k}^{-1}}{2}\right)\Big/{\dfrac{\beta_{k}^{-1}}{2}}=\left(\dfrac{(1-\epsilon_{2})(1-t_k)}{2}+\dfrac{\beta_{k}}{2}\right)\Big/{\dfrac{\beta_{k}}{2}}=\dfrac{1}{q_{1,k}^{2}}.
		\end{equation*}
		Solving for $t_k$ and $q_{1,k}$ yields
		\begin{equation*}
			t_k=\dfrac{L(1-\epsilon_{2})}{L(1-\epsilon_{2})+\beta_{k}^{2}(1-\epsilon_{1})},
			\quad
			q_{1,k}=\sqrt{1\Big/\left(1+\dfrac{(1-\epsilon_{1})(1-\epsilon_{2})\beta_k}{(1-\epsilon_{1})\beta_{k}^{2}+(1-\epsilon_{2})L}\right)}.
		\end{equation*}
		Finally, noting that $\underline{\beta}\leq \beta_k \leq \overline{\beta}$ under Assumption 1, we define
		\[ q_{1}=\sqrt{{1}\Big/\left(1+\dfrac{(1-\epsilon_{1})(1-\epsilon_{2})\underline{\beta}}{(1-\epsilon_{1})\bar{\beta}^{2}+(1-\epsilon_{2})L}\right)}. \]
		It follows that $0 < q_{1,k} \leq q_{1} < 1$ for all $k$, and consequently
		\begin{equation*}
			\|\hat{v}^{k+1}-v^{\ast}\|_{M_{k}}\leq q_{1,k} \|\hat{v}^{k}-v^{\ast}\|_{M_{k}}\leq q_{1} \|\hat{v}^{k}-v^{\ast}\|_{M_{k}},
		\end{equation*}
		which completes the proof.
	\end{proof}

	\subsection{Convergence of inexact framework (\ref{eq:inexact ADMM-u})-(\ref{eq:inexact ADMM-lambda})}\label{sec:Convergence of inexact framework}	
	In this subsection, we present the convergence of the inexact ADMM framework (\ref{eq:inexact ADMM-u})-(\ref{eq:inexact ADMM-lambda}). We first establish  the globally strong convergence of the framework. Furthermore, if the criterion parameters ${\theta_k}$ decay geometrically, we will show that the framework admits a linear convergence rate in the nonergodic sense.

	By using similar arguments for Lemma \ref{lem: exact KKT est}, from the definition \eqref{eq:criterion function} and Lemma \ref{lem:estimate for KKT} we have the following result for the solutions of the inexact ADMM \eqref{eq:inexact ADMM-u}--\eqref{eq:inexact ADMM-lambda}.	
	\begin{lemma}
		Let $\{\omega^{k}\}$ be the sequence generated by the inexact ADMM (\ref{eq:inexact ADMM-u})-(\ref{eq:inexact ADMM-lambda}). Then, for any $k \geq 0$, it holds that
		\begin{equation}\label{eq: desent of KKT strong con inexact}
			\begin{aligned}
				&J(u^{k+1})+R(z^{k+1})-J(u^{\ast})-R(z^{\ast})-\langle \lambda^{\ast} ,u^{k+1}-z^{k+1}\rangle\\
				\leq
				&\frac{1}{2}(\| v^{k}-v^{\ast}\|_{M_{k}}^{2}-\| v^{k+1}-v^{\ast}\|_{M_{k}}^{2}-\| v^{k+1}-v^{k}\|_{M_{k}}^{2}-\|u^{k+1}-u^{\ast}\|^{2})\\
				&+\langle \sigma_{k}(u^{k+1}), u^{k+1}-u^{\ast}\rangle.
			\end{aligned}
		\end{equation}
	\end{lemma}
	
	In addition to the above lemma, the following boundedness result for the sequence $\{v^{k}\}$ plays a key role in  in establishing the convergence of the sequence $\{\omega^{k}\}$.
	\begin{lemma}\label{lem:inexact solution bounded}
		Let $\{{\omega}^{k}\}$ be the sequence generated by the inexact ADMM (\ref{eq:inexact ADMM-u})-(\ref{eq:inexact ADMM-lambda}). Under Assumptions \ref{assume:augmented parameter} and \ref{assume:criterion parameter}, the sequence $\{v^k\}$ is uniformly bounded in the weighted norm. That is, there exists a constant $C_v > 0$, independent of the iteration $k$, such that for all $k \geq 0$,
		\begin{equation}\label{eq: inexact solution bounded}
			\|v^{k+1} - v^{\ast}\|_{M_{k+1}} \leq C_v.
		\end{equation}
	\end{lemma}
	\begin{proof}
		From (\ref{eq:inexact ADMM-z}), (\ref{eq:inexact ADMM-lambda}), (\ref{eq: OptCondition of auxiliary ADMM-z}) and (\ref{eq: OptCondition of auxiliary ADMM-lambda}), it follows that		
		\begin{align}
			&\| v^{k+1}-\tilde{v}^{k+1}\|_{M_{k+1}}^2 \nonumber \\
			=&\beta_{k+1}\| z^{k+1}-\tilde{z}^{k+1}\|^{2}+
			\beta_{k+1}^{-1}\| \lambda^{k+1}-\tilde{\lambda}^{k+1}\|^{2} \nonumber\\
			= &\beta_{k+1}\| z^{k+1}-\tilde{z}^{k+1}\|^{2}
			+\beta_{k+1}^{-1}\|\beta_{k}(\tilde{u}^{k+1}-\tilde{z}^{k+1})-\beta_{k}(u^{k+1}-z^{k+1})\|^{2}  \nonumber\\
			\leq & (\beta_{k+1}+2\beta_{k}^{2}\beta_{k+1}^{-1})\| z^{k+1}-\tilde{z}^{k+1}\|^{2}+2\beta_{k}^{2}\beta_{k+1}^{-1}\| \tilde{u}^{k+1}-u^{k+1}\|^{2}\label{eq:auxiliary convergence I 1}\\
			\leq &
			(\beta_{k+1}+2\beta_{k}^{2}\beta_{k+1}^{-1})\left\|P_{\mathcal{C}}\left(S_{\frac{\gamma_{s}}{\beta_{k}}}(u^{k+1} - \beta^{-1}_{k}\lambda^{k} )\right)-P_{\mathcal{C}}\left(S_{\frac{\gamma_{s}}{\beta_{k}}}(\tilde{u}^{k+1} - \beta^{-1}_{k}\lambda^{k} )\right) \right\|^{2} \nonumber\\
			&+
			2\beta_{k}^{2}\beta_{k+1}^{-1}\|\tilde{u}^{k+1}-u^{k+1}\|^{2} \nonumber\\
			\leq &  \max\limits_{k}\left\{\beta_{k+1}+4\beta_{k}^{2}\beta_{k+1}^{-1}\right\} \|u^{k+1} - \tilde{u}^{k+1}\|^{2}
			\leq C_{\beta}^2\theta_{k}^{2},  \nonumber
		\end{align}
		with $C_{\beta}=\sqrt{\bar{\beta}+4\bar{\beta}^{2}\underline{\beta}^{-1}}$.  The second-to-last inequality comes from the nonexpansivity of $P_{\mathcal{C}}$ and $ S_{\frac{\gamma_{s}}{\beta_{k}}}$. Applying Theorem \ref{thm: exact Convergence rate} to the auxiliary solutions and inexact solutions yields
		\begin{equation*}
			\|\tilde{v}^{k+1}-v^{\ast}\|_{M_{k}}\leq q_{1} \|v^{k}-v^{\ast}\|_{M_{k}}.
		\end{equation*}
		Then, we have
		\begin{align}
			&\| v^{k+1}-v^{\ast}\|_{M_{k+1}} \nonumber \\
			\leq& \|v^{k+1}-\tilde{v}^{k+1}\|_{M_{k+1}}+\| \tilde{v}^{k+1}-v^{\ast}\|_{M_{k+1}} \nonumber \\
			\leq &\|v^{k+1}-\tilde{v}^{k+1}\|_{M_{k+1}}+(1+\eta_{k})\| \tilde{v}^{k+1}-v^{\ast}\|_{M_{k}}\nonumber \\
			\leq &C_{\beta}\theta_{k} +(1+\eta_{k})q_{1}\|v^{k}-v^{\ast} \|_{M_{k}} \nonumber \\
			\leq& C_{\beta} \theta_{k}+ (1+\eta_{k})q_{1}((1+\eta_{k-1})q_{1}\|v^{k-1}-v^{\ast}\|_{M_{k-1}}+C_{\beta} \theta_{k-1}) \label{eq: inexact solution bounded_1} \\
			\leq& \left(\prod_{k=1}^{\infty}(1+\eta_{k})\right)q_{1}^{k+1}\|v^{0}-v^{\ast}\|_{M_{0}}+C_{\beta}\left(\prod_{k=1}^{\infty}(1+\eta_{k})\right)\sum_{i=0}^{k}q_{1}^{k-i}\theta_{i} \nonumber \\
			\leq& \left(\prod_{k=1}^{\infty}(1+\eta_{k})\right) q_{1}^{k+1}\|v^{0}-v^{\ast}\|_{M_{0}}+C_{\beta}\left(\prod_{k=1}^{\infty}(1+\eta_{k})\right) \sum_{i=0}^{k} \theta_{i}
			<\infty, \nonumber
		\end{align}
		which establishes the result (\ref{eq: inexact solution bounded}).
	\end{proof}

	Now, we are in a position to establish global strong convergence of the inexact ADMM framework \eqref{eq:inexact ADMM-u}--\eqref{eq:inexact ADMM-lambda}.
	\begin{theorem}\label{thm: inexact convergence}
		Let $\{{\omega}^{k}\}$  be the sequences generated by (\ref{eq:inexact ADMM-u})-(\ref{eq:inexact ADMM-lambda}). Under Assumption 1 and Assumption 2, it holds that
		\begin{equation}\label{eq: inexact convergence of w}
			\lim\limits_{k\rightarrow \infty} \| \omega^{k+1}-\omega^{\ast}\|= 0.
		\end{equation}
	\end{theorem}
	\begin{proof}
		According to Assumption 1 and the definition of  $\|\,\cdot \,\|_{M_{k}}$ in (\ref{eq:induced norm}), we have
		\begin{equation}\label{eq:Assump 1 vk-v*}
			\| v^{k+1}-v^{\ast}\|_{M_{k+1}}^{2}-\| v^{k+1}-v^{\ast}\|_{M_{k}}^{2}\leq \eta_{k}\| v^{k+1}-v^{\ast}\|_{M_{k}}^{2}.
		\end{equation}
		Based on the estimation (\ref{eq: inexact solution bounded}), we get
		\begin{equation}\label{eq: uniform_bound}
			\begin{aligned}
				\|v^{k+1}-v^{\ast}\|\leq &\,\sqrt{\max\{ \bar{\beta}, \underline{\beta}^{-1}\}}\|v^{k+1}-v^{\ast}\|_{M_{k+1}}\leq \sqrt{\max\{ \bar{\beta}, \underline{\beta}^{-1}\}} C_{v},\\
				\| v^{k+1}-v^{\ast}\|_{M_{k}}\leq&\, \sqrt{\max\{ \bar{\beta}, \underline{\beta}^{-1}\}}\|v^{k+1}-v^{\ast}\|\leq \max\{ \bar{\beta}, \underline{\beta}^{-1}\} C_{v},\\
				\|u^{k+1}-u^{\ast}\|\leq&\, \|u^{k+1}-z^{k+1}\|+\|z^{k+1}-z^{\ast}\|={\beta_{k}}^{-1}\|\lambda^{k+1}-\lambda^{k}\|+\|z^{k+1}-z^{\ast}\|\\
				\leq& {\beta_{k}}^{-1}(\|\lambda^{k+1}-\lambda^{\ast}\|+\|\lambda^{k}-\lambda^{\ast}\|)+\|z^{k+1}-z^{\ast}\| \\
				\leq& \max\{2\underline{\beta}^{-1},1\}\sqrt{\max\{ \bar{\beta}, \underline{\beta}^{-1}\}} C_{v}.
			\end{aligned}
		\end{equation}
		For simplicity, the upper bound of the sequences $\| v^{k+1}-v^{\ast}\|_{M_{k+1}}$, $\|v^{k+1}-v^{\ast}\|$, $\| v^{k+1}-v^{\ast}\|_{M_{k}}$ and $\|u^{k+1}-u^{\ast}\|$ is uniformly denoted by $C_{w}$.
		Further, from (\ref{eq:variational inequality}), (\ref{eq: desent of KKT strong con inexact}) and (\ref{eq:Assump 1 vk-v*}), we have
		\begin{align*}
			0\leq &\| v^{k}-v^{\ast}\|_{M_{k}}^{2}-\| v^{k+1}-v^{\ast}\|_{M_{k}}^{2}-\| v^{k+1}-v^{k}\|_{M_{k}}^{2}-\|u^{k+1}-u^{\ast}\|^{2}+2\langle \sigma_{k}(u^{k+1}), u^{k+1}-u^{\ast}\rangle\\
			=
			&\| v^{k}-v^{\ast}\|_{M_{k}}^{2}-\| v^{k+1}-v^{\ast}\|_{M_{k+1}}^{2}+\| v^{k+1}-v^{\ast}\|_{M_{k+1}}^{2}-\| v^{k+1}-v^{\ast}\|_{M_{k}}^{2}\\
			&-\| v^{k+1}-v^{k}\|_{M_{k}}^{2}-\|u^{k+1}-u^{\ast}\|^{2}+2\langle \sigma_{k}(u^{k+1}), u^{k+1}-u^{\ast}\rangle\\
			\leq
			&\| v^{k}-v^{\ast}\|_{M_{k}}^{2}-\| v^{k+1}-v^{\ast}\|_{M_{k+1}}^{2}+\eta_{k}\| v^{k+1}-v^{\ast}\|_{M_{k}}^{2}\\
			&-\| v^{k+1}-v^{k}\|_{M_{k}}^{2}-\|u^{k+1}-u^{\ast}\|^{2}+2\| \sigma_{k}(u^{k+1})\|\|u^{k+1}-u^{\ast}\|\\
			\leq
			&\| v^{k}-v^{\ast}\|_{M_{k}}^{2}-\| v^{k+1}-v^{\ast}\|_{M_{k+1}}^{2}+\eta_{k}C_{w}^2\\
			&-\| v^{k+1}-v^{k}\|_{M_{k}}^{2}-\|u^{k+1}-u^{\ast}\|^{2}+2C_{w}\theta_{k}.
		\end{align*}
		Summing up the above inequality yields
		\begin{align*}
			\sum_{k=0}^{K}\| v^{k+1}-v^{k}\|_{M_{k}}^{2}+\sum_{k=0}^{K}\|u^{k+1}-u^{\ast}\|^{2}\leq \| v^{0}-v^{\ast}\|_{M_{0}}^{2}+C^2_{w}\left(\sum_{k=0}^{K}\eta_{k}\right)+2C_{w}\left(\sum_{k=0}^{K}\theta_{k}\right)
			<\infty,
		\end{align*}
		which implies that
		\begin{align}\label{eq: boundedness of u}
			\lim\limits_{k\rightarrow \infty} \| u^{k+1}-u^{\ast}\|= 0,\qquad
			\lim\limits_{k\rightarrow \infty} \| v^{k+1}-v^{k}\|=0.
		\end{align}
		From the equation (\ref{eq:inexact ADMM-lambda}) and the boundedness of $\beta_k$, it follows that
		\begin{align*}
			\| z^{k+1}-z^{\ast}\|\leq \|z^{k+1}-u^{k+1}\|+\|u^{k+1}-u^{\ast}\|\leq \underline{\beta}^{-1}\|\lambda^{k+1}-\lambda^{k}\|+\|u^{k+1}-u^{\ast}\|.
		\end{align*}
		Therefore, combining the convergence properties in (\ref{eq: boundedness of u}), we obtain
		\begin{equation*}
			\lim\limits_{k\rightarrow \infty} \| z^{k+1}-z^{\ast}\|= 0.
		\end{equation*}
		In addition, combining the definition of $\sigma_{k}(u)$ in (\ref{eq:criterion function}) with the equation (\ref{eq:inexact ADMM-lambda}),
		and applying the optimality of $\omega^{\ast}$ and the property of $\sigma_{k}(u)$ in (\ref{eq:inexact criterion}), we have
		\begin{align*}
			\|\lambda^{k+1}-\lambda^{\ast}\|=&\|DJ(u^{k+1})+\beta_{k}(z^{k+1}-z^{k})-\sigma_{k}(u^{k+1})-DJ(u^{\ast})\|\\
			\leq & \|DJ(u^{k+1})-DJ(u^{\ast})\|+\beta_{k}\|z^{k+1}-z^{k}\|+\|\sigma_{k}(u^{k+1})\|\\
			\leq & L \|u^{k+1}-u^{\ast}\|+\beta_{k}\|z^{k+1}-z^{k}\|+\theta_{k}.
		\end{align*}
		Combining the convergence properties in (\ref{eq: boundedness of u}) and Assumption 2, we have
		\begin{equation*}
			\lim\limits_{k\rightarrow \infty} \| \lambda^{k+1}-\lambda^{\ast}\|= 0.
		\end{equation*}
		This completes the proof.
	\end{proof}
	
	With a suitable choice of $\theta_k$, we obtain a linear rate of convergence for the inexact ADMM framework  \eqref{eq:inexact ADMM-u}--\eqref{eq:inexact ADMM-lambda}.
	
	\begin{theorem}
		Let $\{{\omega}^{k}\}$  be the sequence generated by (\ref{eq:inexact ADMM-u})-(\ref{eq:inexact ADMM-lambda}) with $ \theta_{k}=q_{2}^{k}$, where $q_{2}\in (0,1)$ is a given constant distinct from $q_{1}$. Under Assumption 1, there exist positive constants $ C $ and $q = \max(q_1, q_2) \in (0,1)$ such that
		\begin{equation*}
			\| v^{k+1}-v^{\ast}\|_{M_{k+1}}\leq  C q^{k+1}.
		\end{equation*}
	\end{theorem}
	\begin{proof}
		Substituting $ \theta_{k}=q_{2}^{k}$ into the second-to-last inequality in (\ref{eq: inexact solution bounded_1}), we obtain
		\begin{equation*}
			\begin{aligned}
				\| v^{k+1}-v^{\ast}\|_{M_{k+1}}
				&\leq \left(\prod_{k=1}^{\infty}(1+\eta_{k})\right)q_{1}^{k+1}\|v^{0}-v^{\ast}\|_{M_{0}}+C_{\beta}\left(\prod_{k=1}^{\infty}(1+\eta_{k})\right)\sum_{i=0}^{k}q_{1}^{k-i}q_{2}^{i}\\
				&\leq \left(\prod_{k=1}^{\infty}(1+\eta_{k})\right) \left(\|v^{0}-v^{\ast}\|_{M_{0}}+C_{\beta}|q_{1}-q_{2}|^{-1} \right)q^{k+1},
			\end{aligned}
		\end{equation*}
		which completes the proof.
	\end{proof}
	
	\begin{remark}
		Although the choice $\theta_k = q_2^k$ guarantees linear convergence of the proposed framework, it may not be the most computationally efficient strategy in practice. There is an inherent trade-off between the inner-iteration accuracy and outer-iteration convergence speed: solving the subproblem more accurately (smaller $\theta_k$) yields a convergence speed closer to that of the exact algorithm, but at the cost of increased inner-loop computational effort. Consequently, the overall computation time may not be minimized. Therefore, striking an appropriate balance between inner accuracy and outer convergence speed remains an important and challenging issue. Besides the geometric decay $\theta_k = q_2^k$, another admissible choice satisfying Assumption 2 is the algebraic decay $\theta_k = 1/k^{\alpha}$ with $\alpha > 1$. The relative effectiveness of these inexact stopping strategies is problem-dependent and requires careful consideration in implementation; see the numerical results shown in next section.
	\end{remark}
	\section{Numerical experiments}\label{sec:numerical experiments}
	In this section, we present numerical results to demonstrate the efficiency of the inexact ADMM framework (\ref{eq:inexact ADMM-u})–(\ref{eq:inexact ADMM-lambda}), referred to as InADMM, on several instances of the parabolic distributed optimal control problem (\ref{eq:obj}).
	
	\subsection{Implementation of the framework (\ref{eq:inexact ADMM-u})-(\ref{eq:inexact ADMM-lambda})}\label{sec:implementation}
	In this subsection, we describe in detail the implementation of the InADMM framework (\ref{eq:inexact ADMM-u})–(\ref{eq:inexact ADMM-lambda}) and present a practical algorithm for the numerical solution of the model problem (\ref{eq:obj})–(\ref{eq:control constraint}).
	
	Recalling the definition of $\bar{S}^{\ast}$, the evaluator $\sigma_{k}(u^{k+1})$ can be rewritten as
	\begin{equation*}\label{eq: reformulation of sigma}
		\begin{aligned}
			\sigma_{k}(u^{k+1}) = ((1+\beta_{k})I  + \gamma_{d} \bar{S}^{\ast}\bar{S})u^{k+1} - ( \beta_{k} z^{k}  + \lambda^{k})
			+ \gamma_{d}\bar{S}^{\ast}( S(0) - y_{d}),
		\end{aligned}
	\end{equation*}
	which can be expressed as
	\begin{equation*}\label{eq: redefinition of sigma}
		\sigma_{k}(u^{k+1}) = H^{k}u^{k+1} -d^{k},
	\end{equation*}
	where $ H^{k} =  (1+\beta_{k})I  + \gamma_{d}\bar{S}^{\ast}\bar{S}$ and $d^{k} =  \beta_{k} z^{k}  + \lambda^{k}-  \gamma_{d}\bar{S}^{\ast}( S(0) - y_{d} )$.
	Thus we can reformulate the inexact criterion (\ref{eq:inexact criterion}) as
	\begin{equation*}\label{eq: reformulation of inexact criterion}
		\|H^{k}u^{k+1} -d^{k}\|\leq \theta_{k}.
	\end{equation*}
	
	It is easy to check that the operator $H^{k}\in \mathcal{L}(U,Y) $ is self-adjoint and positive definite. Then for solving the $u$-subproblem, we can apply the CG method as described in Algorithm~\ref{alg: inner CG}.
	\begin{algorithm}[htbp]
		\caption{ CG for the $ u$-subproblem}
		\label{alg: inner CG}
		
		\textbf{Input:} 
		$ u^{(0)}=u^{k} $, $ m=0 $, $r^{(0)}=-\sigma_{k}(u^{(0)})$, $q^{(0)}=r^{(0)}$.
		
		\textbf{do}
		\begin{algorithmic}[1]
			\item Compute $u^{(m+1)}$ by
			\begin{equation*}
				u^{(m+1)}=u^{(m)}+\alpha_{m}q^{(m)},\quad\text{with}\quad
				\alpha_{m}=\dfrac{\langle r^{(m)},q^{(m)}\rangle}{\langle H^{k}q^{(m)},q^{(m)}\rangle},
			\end{equation*}			
			
			\item Compute $r^{(m+1)}$ by
			\begin{equation*}
				r^{(m+1)}=r^{(m)}-\alpha_{m}H^{k}q^{(m)},
			\end{equation*}
			
			\item Compute $q^{(m+1)}$ by
			\begin{equation*}
				q^{(m+1)}=r^{(m+1)}+\rho_{m}q^{(m)} \quad \text{with} \quad
				\rho_{m}=\dfrac{\langle r^{(m+1)},r^{(m+1)}\rangle}{\langle r^{(m)},r^{(m)}\rangle}.
			\end{equation*}	
		\end{algorithmic}
		\textbf{while} $ \| \sigma_{k}(u^{(m+1)})\|\leq\theta_{k} $.
		
		\textbf{Output}
		\begin{equation*}
			u^{k+1} = u^{(m+1)}.
		\end{equation*}
	\end{algorithm}
	
	In the $u$-subproblem of the framework \eqref{eq:inexact ADMM-u}--\eqref{eq:inexact ADMM-lambda}, the operator $H^{k}$ is almost unchanged throughout the iterations, with the only variation coming from the parameter $\beta_{k}$, which in turn enables an efficient execution of Algorithm~1. Now, based on the above discussion of the $u$-subproblem, we can  instantiate the InADMM framework (\ref{eq:inexact ADMM-u})–(\ref{eq:inexact ADMM-lambda}) in the following Algorithm~\ref{alg: InADMM}.
	
	\begin{algorithm}[htbp]
		\caption{ InADMM for the optimal control problem $(\ref{eq:obj})-(\ref{eq:control constraint})$}
		\label{alg: InADMM}
		\textbf{Input:} 
		$ \omega^{0}=(u^{0},z^{0},\lambda^{0})^{T}$, $\{\beta_{k}\}_{0}^{\infty}$, and $\{\theta_{k}\}_{0}^{\infty} $.
		
		\textbf{for} $ k\geq 0 $
		\begin{algorithmic}[1]
			\item Update $u^{k+1}$ by Algorithm 1 with $\theta_{k}$.
			\item Update $z^{k+1}$ by	
			\begin{equation*}
				z^{k+1}= P_{\mathcal{C}}\left(S_{\frac{\gamma_{s}}{\beta_{k}}}\left(u^{k+1} - \beta^{-1}_{k}\lambda^{k}\right)\right).
			\end{equation*}
			\item Update $\lambda^{k+1}$ by
			\begin{equation*}
				\lambda^{k+1}=\lambda^{k}-\beta_{k}(u^{k+1}-z^{k+1}).
			\end{equation*} 		
		\end{algorithmic}
		\textbf{end for}
	\end{algorithm}
	
	In practice, the value of penalty parameter $\beta_{k}$ usually has an important effect on the behavior of the algorithm: values that are too large or too small tend to slow down convergence. In our numerical implementation, we adopt the strategy in \cite{he2002new} to adjust $\beta_{k}$ as follows:
	\begin{equation*}\label{eq:update of beta}
		\beta_{k+1}=\begin{cases}
			(1+\eta_{k})\beta_{k}      &\text{if}
			\quad \beta_{k-1} \|z^{k-1} -z^{k}\| < \frac{1}{4}\|u^{k}-z^{k}\|, \\
			\frac{\beta_{k}}{1+\eta_{k}}  &\text{if}
			\quad \beta_{k-1} \|z^{k-1} -z^{k}\| > \frac{1}{4}\|u^{k}-z^{k}\|,\\
			\beta_{k} &\text{otherwise},
		\end{cases}
	\end{equation*}
	which ensures that the sequence ${\beta_k}$  satisfies Assumption 1.
	
	\subsection{Experiment settings}\label{sec:experiment settings}
	To demonstrate the efficiency of our method, in consideration of the feasible iterative $\{z^{k}\}$, we define the state relative distance SRD and the objective value Obj as
	\begin{equation*}
		\text{SRD} = \frac{\|S(z^{k})-y_{d}\|}{\|y_{d}\|},\quad \text{and} \quad
		\text{Obj} =  \dfrac{\gamma_{d}}{2}\| S(z^{k})-y_{d}\|_{L^{2}(Q)}^2+\dfrac{1}{2}\| z^{k}\|_{L^{2}(G)}^2+\gamma_{s}\| z^{k}\|_{L^{1}(G)},
	\end{equation*}
	respectively. The primal residual PR and dual residual DR are defined as
	\begin{equation*}
		\text{PR} = \frac{\beta_{k}\|z^{k}-z^{k-1}\|}{\|z^{k-1}\|}, \quad \text{and} \quad
		\text{DR} = \frac{\|u^{k} - z^{k}\|}{\max\{\|u^{k}\|,\|z^{k}\|\}},
	\end{equation*}
	respectively.  We let {CG-Ave/Max} denote the average and maximum steps of the CG method.
	
	
	For all experiments, we adopt
	\begin{equation*}
		\max\left\{\text{PR},\text{DR}\right\}\leq \text{tol}
	\end{equation*}
	as the stopping criterion uniformly, where $\text{tol} =10^{-4}$. In addition,  the subproblem \eqref{eq:inexact ADMM-u} is regarded as solved ``exactly'' when the tolerance $\theta_k$ falls below $10^{-6}$. The maximum number of outer iterations is set to $10^{3}$.
	
	All numerical simulations  are started from a zero initial guess. For numerical discretization, we use linear finite elements on a uniform spatial mesh of size $h$ and the backward Euler scheme with temporal step size $\tau$. All experiments were performed on a desktop computer with a 13th Gen Intel Core i7-13700K 3.40 GHz processor.
	
	\subsection{Numerical results for the non-sparse case}\label{sec:smooth case}
	
	In this subsection, we test Algorithm 2 on the constrained parabolic optimal distributed control problem without sparsity promotion, i.e., $\gamma_{s}=0$. In the first example, the control is active on the entire domain, whereas in the second it is active only on a subdomain.
	
	We consider the example given in \cite{glowinski2022application}:
	\begin{equation*}
		\min\limits_{u}\ J(u)=\dfrac{\gamma_{d}}{2}\diint_{Q}|y(u)-y_{d}|^{2}{\rm d}x{\rm d}t+\dfrac{1}{2}\diint_{G}|u|^{2}{\rm d}x{\rm d}t,
	\end{equation*}
	where the domain $\Omega = (0,1)^{2}$, $T=1$ and the state equation reads
	\begin{equation}\label{eq: general state eq of example 1}
		\left\{
		\begin{aligned}
			&\frac{\partial y}{\partial t}-\Delta y=f+\chi_{G}u,&\quad&\text{\,in} \quad\Omega\times(0,T),\\
			&y=0,&\quad  &\text{on}\quad\Gamma\times (0,T),\\
			&y(0)=\phi ,&\quad&\text{\,in} \quad\Omega.
		\end{aligned}
		\right.
	\end{equation}

	{\bf Example~1~Entire domain control.} Our first experiment was carried out with the  subdomain $\Omega_{sub} = \Omega$, the penalty factor $\gamma = 10^{5}$, and the control constraint
	\begin{equation*}
		\mathcal{C}=\left\{u|u\in  L^{2}(G),~-0.5 \leq u(x,t) \leq 0.5~\text{a.e.~in}~\Omega\times (0,T) \right\}.
	\end{equation*}
	Following the construction in \cite{glowinski2022application}, we let
	\begin{equation*}
		\begin{aligned}
			\bar{y}=&(1-t)\sin\pi x_{1}\sin\pi x_{2},
			&&\bar{p}=\dfrac{1}{\gamma_{d}}(1-t)\sin2\pi x_{1}\sin2\pi x_{2},\\
			\bar{u}=&\min(a,\max(b,-\gamma_{d}\bar{p})),
			&&y_{d}=\bar{y}+\frac{\partial \bar{p}}{\partial t}+\Delta \bar{p},\\
			f=&-\bar{u}+\frac{\partial \bar{y}}{\partial t}-\Delta \bar{y},
			&&\phi=\sin\pi x_{1}\sin\pi x_{2}.
		\end{aligned}
	\end{equation*}
	Then, with the source function $f$, the optimal control $u^{\ast} = \bar{u}$ and the corresponding optimal state $y^{\ast} = \bar{y}$.

	In this experiment, the parameters in Algorithm 2 were set as $\eta_{k} = 1/2^{k}$, $\beta_{0}=2$, $\beta_{1}=3$, and $\theta_{k} = \theta_{0}/2^{k}$ with $\theta_{0} = \|\sigma_{0}(u^{0})\|/2$. With a mesh size of $h = \tau = 2^{-6}$, we first observe that Algorithm~2 attains the stopping criterion after 20 outer ADMM iterations.
	Figure~\ref{fig:figure1uy}  plots the numerical solutions of $u$ and $y$, and the errors between the numerical solutions and the exact solutions. From the figure, we see that the numerical solutions approximates the exact solutions well.
	
	\begin{figure}[H]
		\centering
		
		\subfigure[Control $u$]{\includegraphics[width=0.25\linewidth]{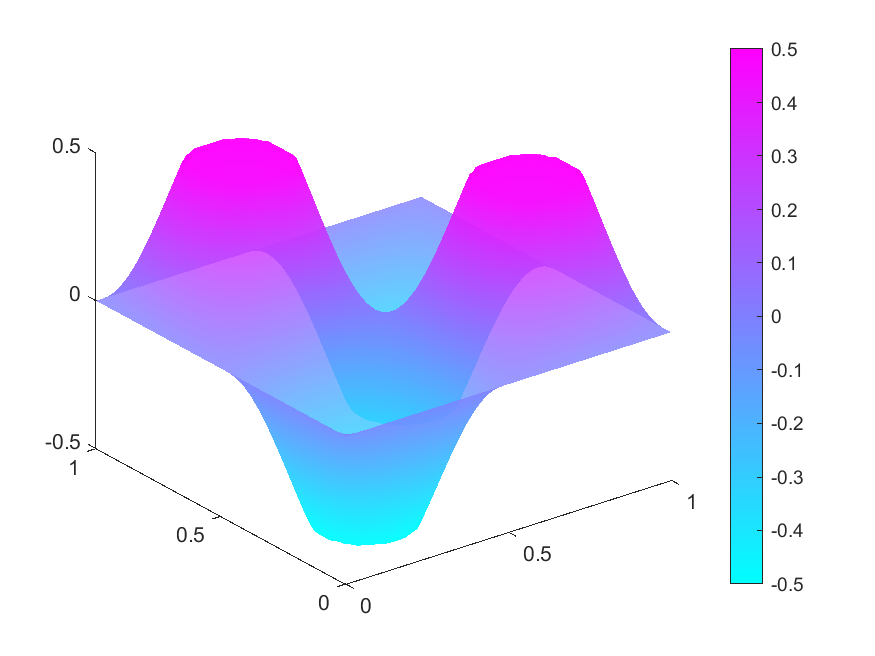}}\hfill
		\subfigure[Error $u-u^{*}$]{\includegraphics[width=0.25\linewidth]{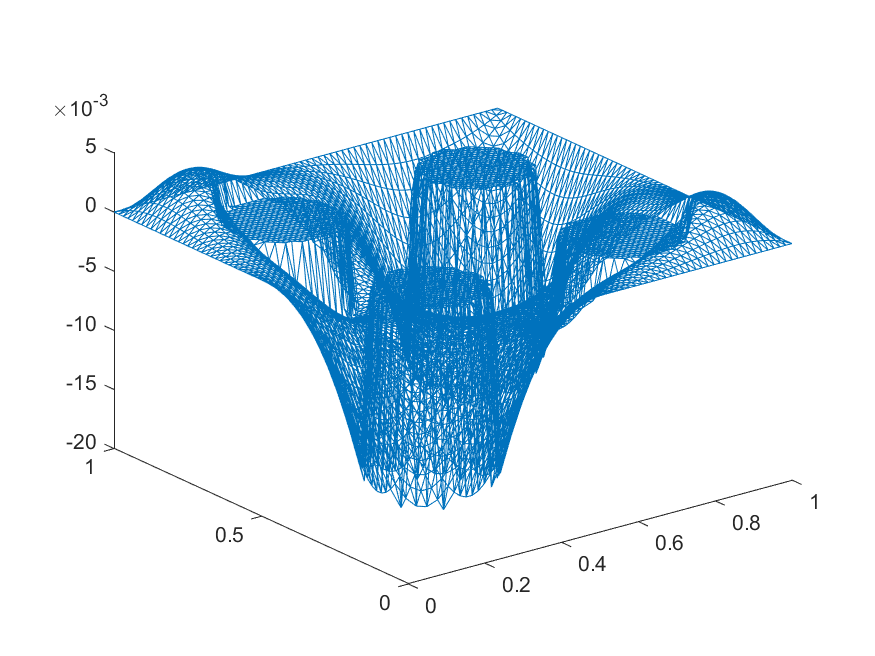}}\hfill
		\subfigure[State $y$]{\includegraphics[width=0.25\linewidth]{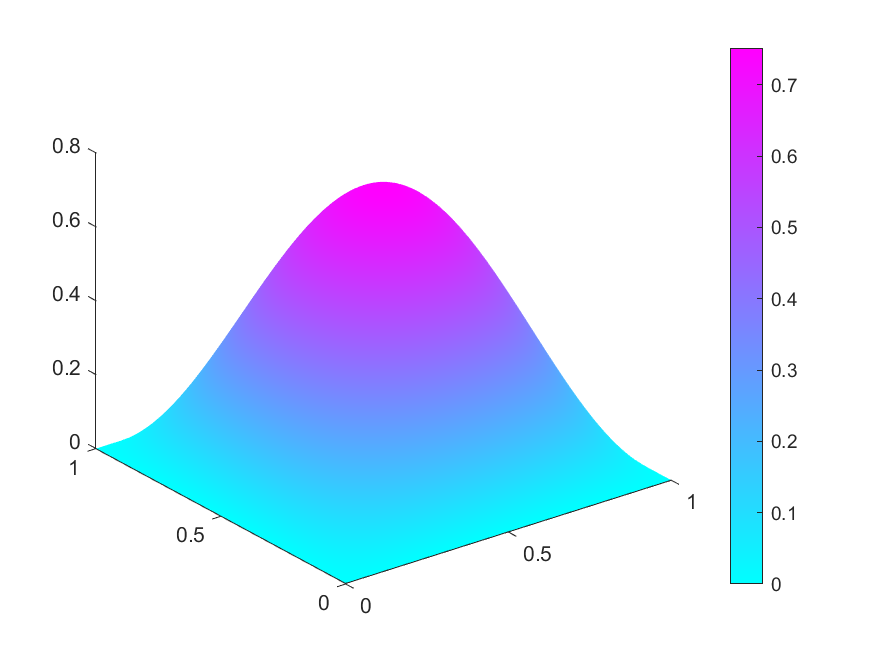}}\hfill
		\subfigure[Error $y-y^{*}$]{\includegraphics[width=0.25\linewidth]{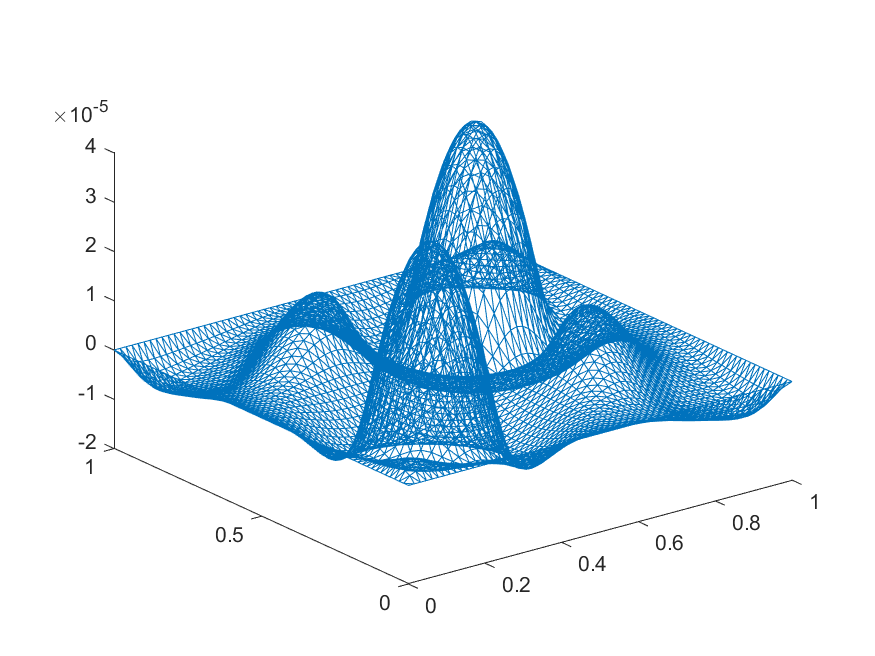}}
		
		\caption{Numerical solutions and errors of $u$ and $y$ at $t=0.25$ in Example 1.}
		\label{fig:figure1uy}
	\end{figure}

	Next, we compare the performance of our method with the classical ADMM and the PGD method \cite{troeltzsch2010optimal} on different mesh sizes. Particularly, as mentioned in Remark 2, we run Algorithm~2 with two different choices of $\theta_k$, namely $\theta_{0}/2^{k}$ and $\theta_{0}/k^{3}$. For the classical ADMM, the tolerance for the inner CG iteration is set to $10^{-6}$. For the PGD method, the step size is determined by a backtracking strategy based on the Armijo-Goldstein condition. The numerical results are reported in Table~\ref{lab:exp1}, where the symbol ``--'' indicates that the PGD method failed to converge within the allowed maximum number of iterations. From the results, we see that both InADMM and ADMM with CG (ADMMCG) outperformed PGD in terms of iteration count and running time. Notably, the proposed InADMM achieves the best performance, attaining the same accuracy with fewer inner iterations and less running time.
	
	\begin{table}[H]
		\centering
		\caption{Numerical results of InADMM, ADMM-CG and PGD in Example~1}
		\scalebox{0.7}{
			\begin{tabular}{|c|c|c|c|c|c|c|c|c|}
				\hline
				\multicolumn{1}{|c|}{Method} & $\theta_{k}$ & ($h,\tau$) & Iteration & $ \| u-u^{*}\|_{L^{2}(G)}$  & Time (s) & CG-Ave/Max & Obj ($\times 10^{-2}$) & SRD ($\times10^{-4}$)\\
				\hline	
				\multirow{6}{*}{InADMM} &
				\multirow{3}{*}{$ \dfrac{\theta_{0}}{k^{3}}$} & $(2^{-6},2^{-6})$  & $19$  &  $4.71\times 10^{-3}$  & $7.44$ & $ 5.68/8 $ &$  3.36  $& $ 7.94 $\\
				\cline{3-9}
				&  & $(2^{-7},2^{-7})$  & $19$  &  $1.19\times 10^{-3}$  & $82.95$ & $ 5.79/8 $ &$  3.40  $& $ 8.01 $\\
				\cline{3-9}
				&  & $(2^{-8},2^{-8})$  & $20$  &  $2.99\times 10^{-4}$  & $1011.66$ & $ 5.10/8 $ &$  3.41  $& $ 8.04 $\\
				\cline{2-9}
				&
				\multirow{3}{*}{$ \dfrac{\theta_{0}}{2^{k}}$} & $(2^{-6},2^{-6})$ & $20$  &  $4.71\times 10^{-3}$  & $10.11$ & $ 8.55/15 $ &$  3.36  $& $ 7.94 $\\
				\cline{3-9}
				&  & $(2^{-7},2^{-7})$  & $20$  &  $1.20\times 10^{-3}$  & $120.65$ & $ 8.65/15 $ &$  3.40  $& $ 8.02 $\\
				\cline{3-9}
				&  & $(2^{-8},2^{-8})$  & $21$  &  $2.96\times 10^{-4}$  & $1642.39$ & $ 9.14/16 $ &$  3.41  $& $ 8.04 $\\
				\hline	
				\multirow{3}{*}{ADMMCG} &
				\multirow{3}{*}{} & $(2^{-6},2^{-6})$  & $15$  &  $4.69\times 10^{-3}$  & $25.84$ & $ 31.13/52 $ &$  3.36  $& $ 7.94 $\\
				\cline{3-9}
				&  & $(2^{-7},2^{-7})$  & $15$  &  $1.18\times 10^{-3}$  & $261.27$ & $ 29.13/51 $ &$  3.40  $& $ 8.02 $\\
				\cline{3-9}
				&  & $(2^{-8},2^{-8})$  & $15$  &  $2.88\times 10^{-4}$  & $2977.55$ & $ 27.87/48 $ &$  3.41  $& $ 8.03 $\\
				\hline
				\multirow{3}{*}{PGD} &
				\multirow{3}{*}{$ $} & $(2^{-6},2^{-6})$  & $288$  &  $5.77\times 10^{-3}$  & $86.57$ & $  $ &$  3.36  $& $ 7.95 $\\
				\cline{3-9}
				&  & $(2^{-7},2^{-7})$  & $700$  &  $1.24\times 10^{-3}$  & $1907.75$ & $  $ &$  3.40  $& $ 8.02 $\\
				\cline{3-9}
				&  & $(2^{-8},2^{-8})$  & $-$  &  $-$  & $-$ & $  $ &$  -  $& $ - $\\
				\hline			
				
			\end{tabular}
		}
		\label{lab:exp1}
	\end{table}
	
	Moreover, the computational efficiency of InADMM with $\theta_k=\theta_0/k^3$ is clearly superior to that with $\theta_k=\theta_0/2^k$, which is consistent with our discussion in Remark 2. This indicates that although a more accurate solution of the subproblem can lead to a higher convergence rate of the outer iteration, it does not necessarily yield the shortest overall computing time for a given accuracy requirement. Consequently, an inexactness strategy that focuses exclusively on preserving a linear convergence rate of the outer iterations is, in general, not optimal from a computational standpoint.

	Now, we present further numerical results for Algorithm~2 with $\theta_k=\theta_0/k^3$ on different mesh sizes in Figure~\ref{fig:example1case1 obj rel}, which includes the primal residual, dual residual, state relative residual and objective value. From the figure, we see that our method performed consistently well in view of both the residuals and the objective value, and is robust with respect to the mesh sizes.
	
	\begin{figure}[H]
		\centering
		
		\subfigure[Primal residual]{\includegraphics[width=0.23\linewidth]{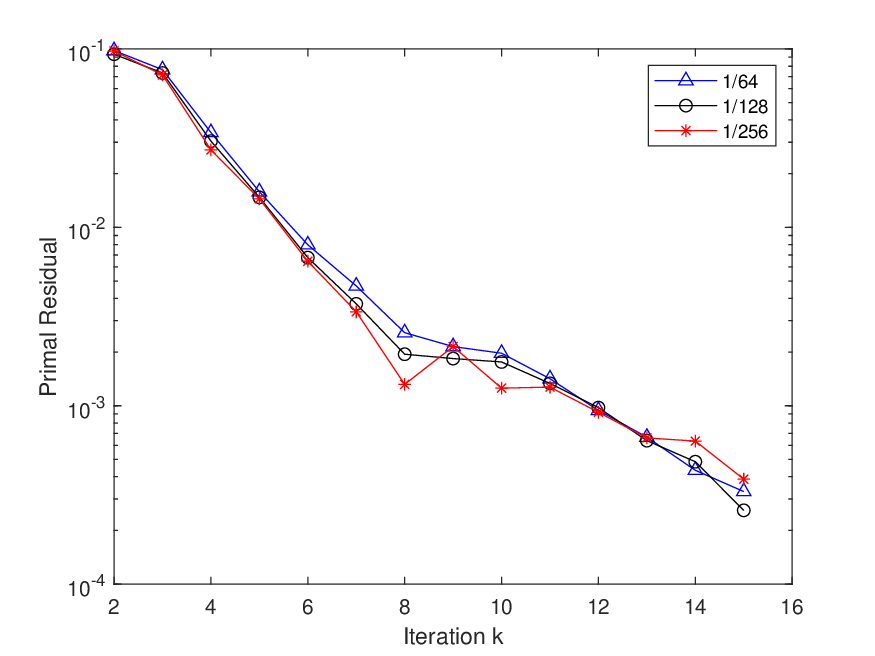}}\hfill
		\subfigure[Dual residual]{\includegraphics[width=0.23\linewidth]{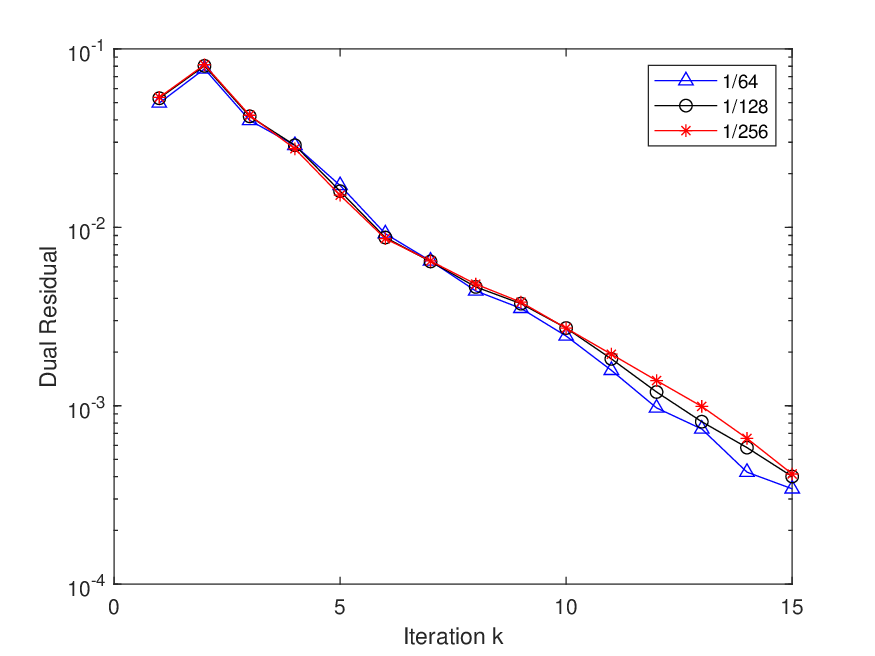}}\hfill
		\subfigure[State distance]{\includegraphics[width=0.23\linewidth]{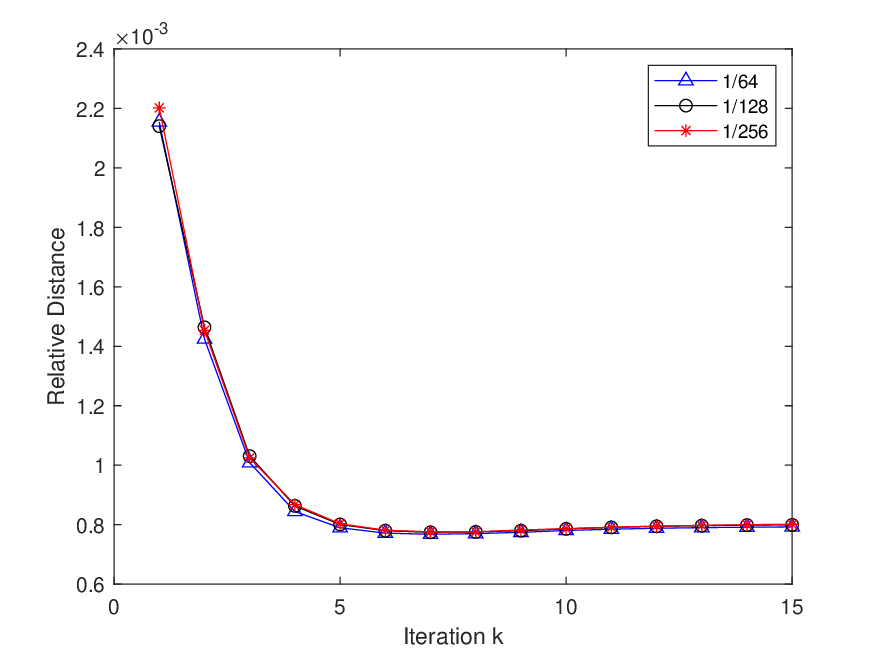}}\hfill
		\subfigure[Objective value]{\includegraphics[width=0.23\linewidth]{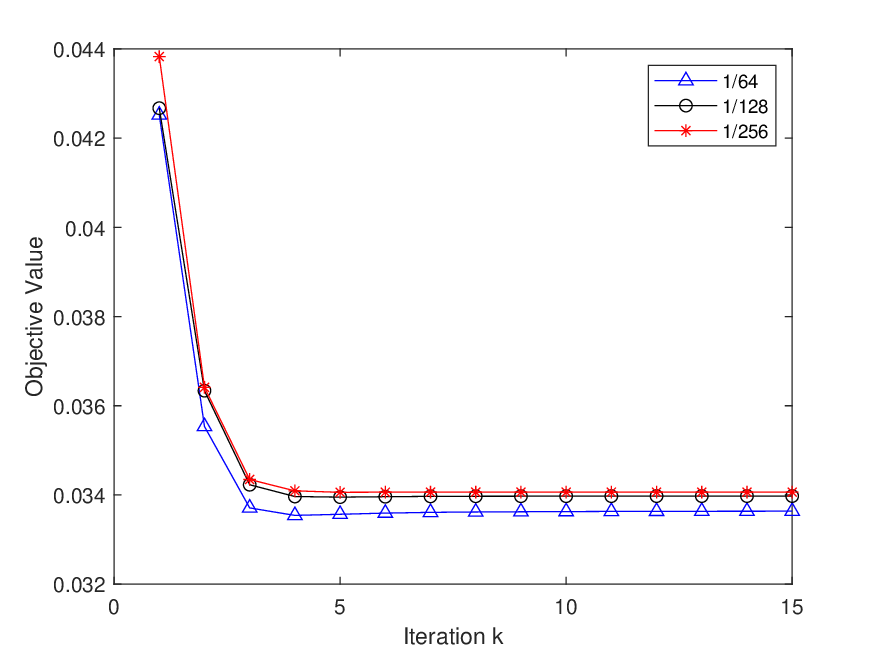}}
		
		\caption{PR, DR, SRD and Obj with different mesh in Example 1.}
		\label{fig:example1case1 obj rel}
	\end{figure}
	
	
	{\bf Example~2~Subdomain control.} 	In the second experiment, we consider the scenario where the control acts on the subdomain $\Omega_{sub} = (0,0.25)^{2}$ with the penalty parameter  $\gamma=10^{6}$ and the constraint $a=-30$ and $b=30$ on the control variable $u$. The target function is set to $y_d = \exp(t) x_1 x_2 (1-x_1)(1-x_2)$. Letting $f = -y$ in \eqref{eq: general state eq of example 1}, we obtain the following state equation:
	\begin{equation*}
		\begin{cases}
			\dfrac{\partial y}{\partial t}-\Delta y+y=\chi_{G}u &\mbox{\,in} \quad\Omega\times(0,T),\\
			y=0                                         &\mbox{on}\quad\Gamma\times (0,T),\\
			y(0)=\phi,                                  &\mbox{\,in} \quad\Omega.
		\end{cases}
	\end{equation*}
	Unlike Example 1, the presence of the reaction term and of the control localized on the subdomain makes an explicit solution unavailable.
	
	
	We ran InADMM with the parameter settings $\beta_{0}=\beta_{1}=8$ and two inexact criteria $\theta_k = \theta_0/k^3$ and $\theta_0/(1.4)^k$, and compare its performance with that of ADMMCG and PGD, which are run with the same settings as in Example~1. The results are reported in Table~\ref{lab:exp2}. Again, the results show that our method achieves the best performance, with a significantly lower computational cost while delivering comparable solution accuracy. Between the two inexact criteria, the geometric decay condition shows a clear advantage in this example. However, as discussed in Remark 2, which of the two inexact criteria is preferable ultimately depends on the specific problem under consideration. We also plot the numerical solution (with $\theta_k=\theta_0/(1.4)^{k}$) of the control $u$ and the state $y$ on the mesh with $h=\tau=2^{-6}$ in Figure~4.
	
	\begin{table}[H]
		\centering
		\caption{Numerical results of InADMM, ADMMCG and PGD in Example~2}
		\scalebox{0.7}
		{
			\begin{tabular}{|c|c|c|c|c|c|c|c|}
				\hline
				\multicolumn{1}{|c|}{Method} & $\theta_{k}$ & ($h,\tau$) & Iteration & Time (s) & CG-Ave/Max & Obj ($\times 10^{3}$) & SRD ($\times10^{-1}$)\\
				\hline	
				\multirow{6}{*}{InADMM} &
				\multirow{3}{*}{$ \theta_{0}/k^{3}$} & $(2^{-6},2^{-6})$  & $34$  &  $6.83$  & $4.56/8$ & $ 1.19$ &$  8.16  $\\
				\cline{3-8}
				&  & $(2^{-7},2^{-7})$  & $34$  &  $60.43$  & $3.94/6$ & $ 1.21 $ &$  8.22 $\\
				\cline{3-8}
				&  & $(2^{-8},2^{-8})$  & $35$  &  $945.93$  & $3.66/5$ & $ 1.22 $ &$  8.25  $\\
				\cline{2-8}
				&
				\multirow{3}{*}{$ \theta_{0}/{(1.4)^{k}}$} &  $(2^{-6},2^{-6})$ & $38$  &  $4.89$  & $2.68/4$ & $ 1.19 $ &$ 8.16  $\\
				\cline{3-8}
				&  & $(2^{-7},2^{-7})$  & $39$  &  $51.26$  & $2.69/5$ & $ 1.21 $ &$  8.22 $\\
				\cline{3-8}
				&  & $(2^{-8},2^{-8})$  & $39$  &  $751.21$  & $2.64/5$ & $ 1.22 $ &$  8.25  $\\
				\hline	
				\multirow{3}{*}{ADMMCG} &
				\multirow{3}{*}{} &  $(2^{-6},2^{-6})$  & $35$  &  $38.65$  & $31.46/79$ & $ 1.19 $ &$  8.16  $\\
				\cline{3-8}
				&  & $(2^{-7},2^{-7})$  & $35$  &  $356.724$  & $22.63/38$ & $ 1.21 $ &$  8.22 $\\
				\cline{3-8}
				&  & $(2^{-8},2^{-8})$  & $35$  &  $4370.60$  & $21.03/30$ & $ 1.22 $ &$  8.25 $\\
				\hline
				\multirow{3}{*}{PGD} &
				\multirow{3}{*}{$ $} &  $(2^{-6},2^{-6})$  & $405$  &  $128.63$  & $ $ & $ 1.19 $ &$  8.18  $\\
				\cline{3-8}
				&  & $(2^{-7},2^{-7})$  & $488$  &  $1476.04 $  & $ $ & $1.21  $ &$  8.24  $ \\
				\cline{3-8}
				&  & $(2^{-8},2^{-8})$  & $-$  &  $-$  & $ $ & $ - $ &  $-$  \\
				\hline			
			\end{tabular}
		}
		\label{lab:exp2}
	\end{table}

	\begin{figure}[H]
		\centering
		
		\subfigure[Control $u$]{\includegraphics[width=0.25\linewidth]{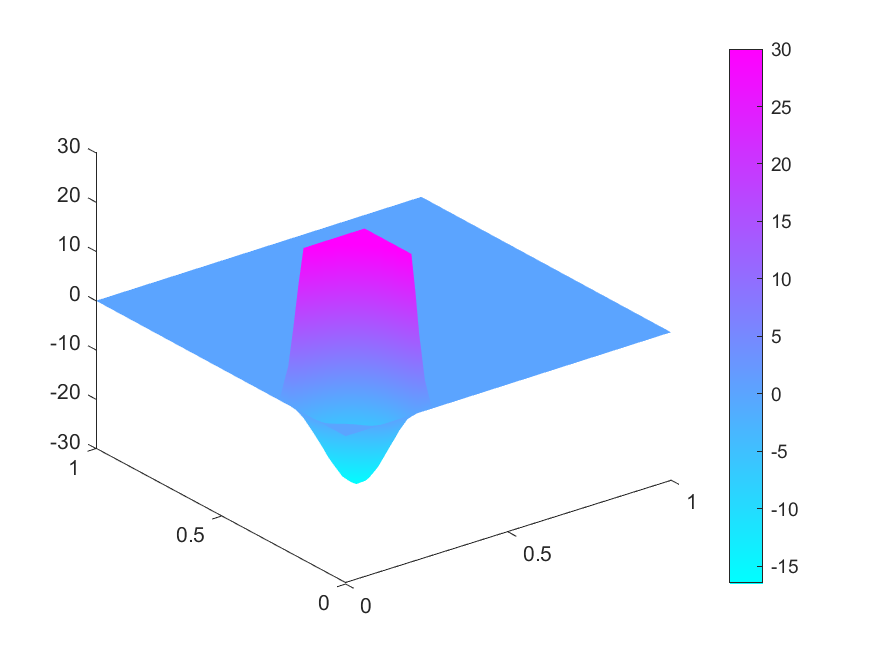}} \hspace{10mm}
		\subfigure[State $y$]{\includegraphics[width=0.25\linewidth]{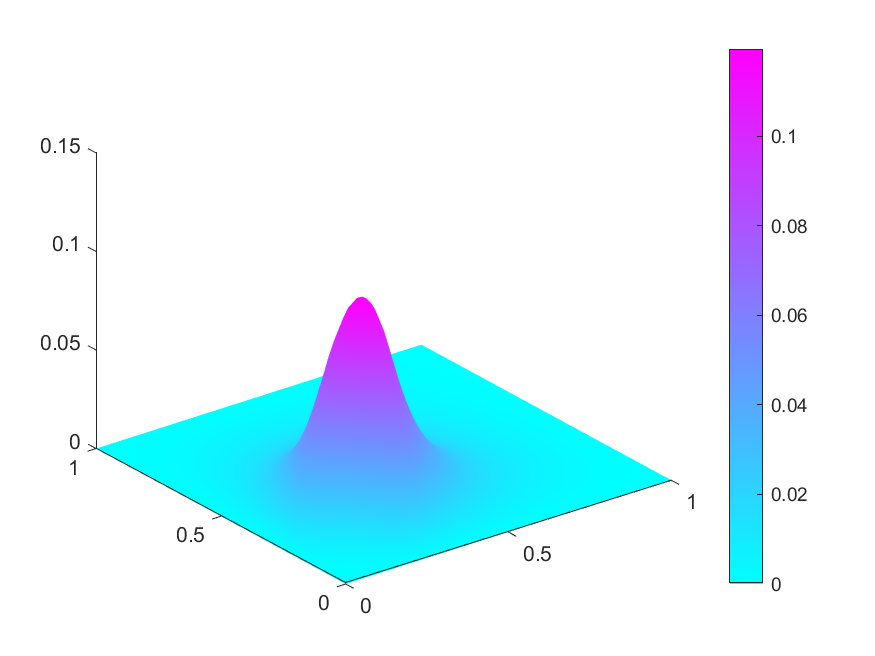}}
		
		\caption{Numerical solutions of the control $u$ and the state $y$ at $t=0.5$ in Example 2.}
		\label{fig: numerical solutions of exp22}
	\end{figure}

	In addition,  we test our method with $\theta_k=\theta_0/(1.4)^k$ on different mesh sizes and present the primal residual, the dual residual, the relative state error, and the objective value in Figure~\ref{fig:example1case2 obj rel}, further confirming that Algorithm 2 converges fast and its performance remains unaffected by mesh sizes.
	
	\begin{figure}[H]
		\centering
		
		\subfigure[Primal residual]{\includegraphics[width=0.23\linewidth]{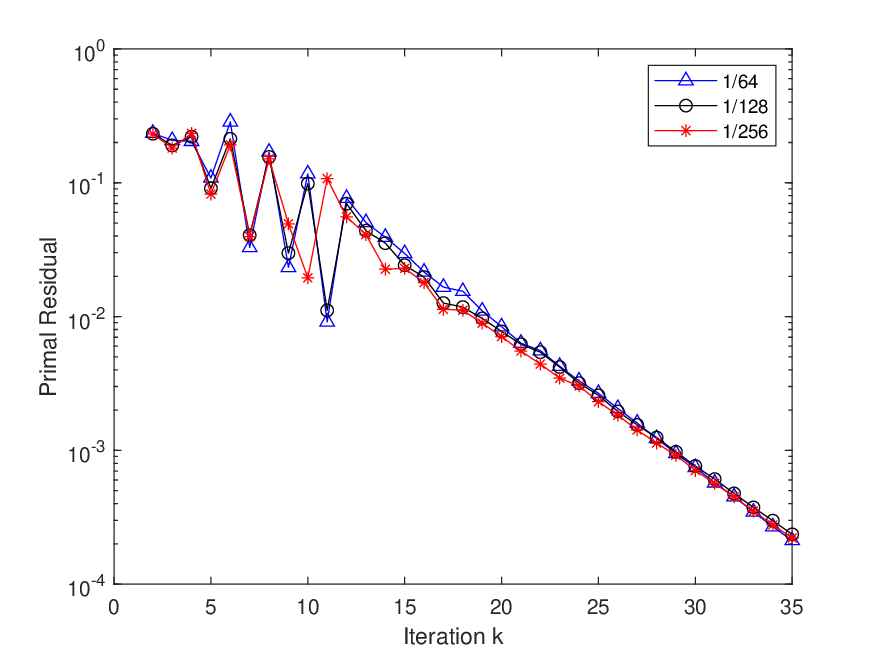}}\hfill
		\subfigure[Dual residual]{\includegraphics[width=0.23\linewidth]{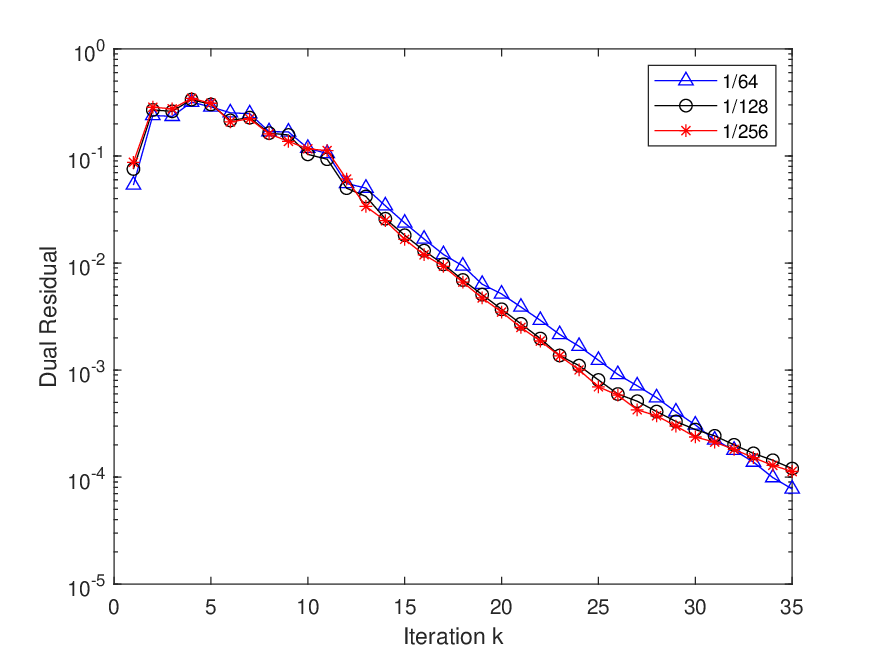}}\hfill
		\subfigure[State distance]{\includegraphics[width=0.23\linewidth]{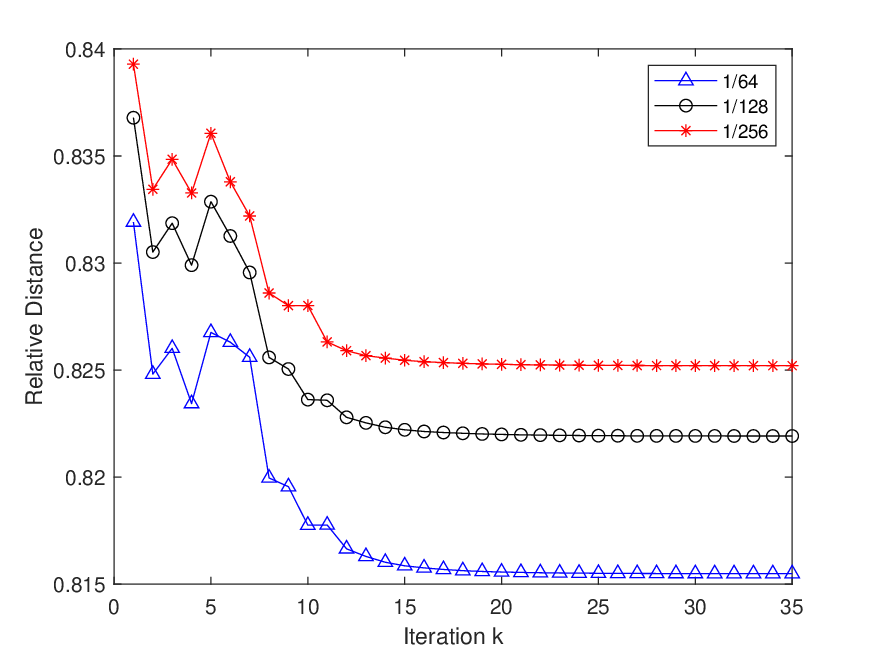}}\hfill
		\subfigure[Objective value]{\includegraphics[width=0.23\linewidth]{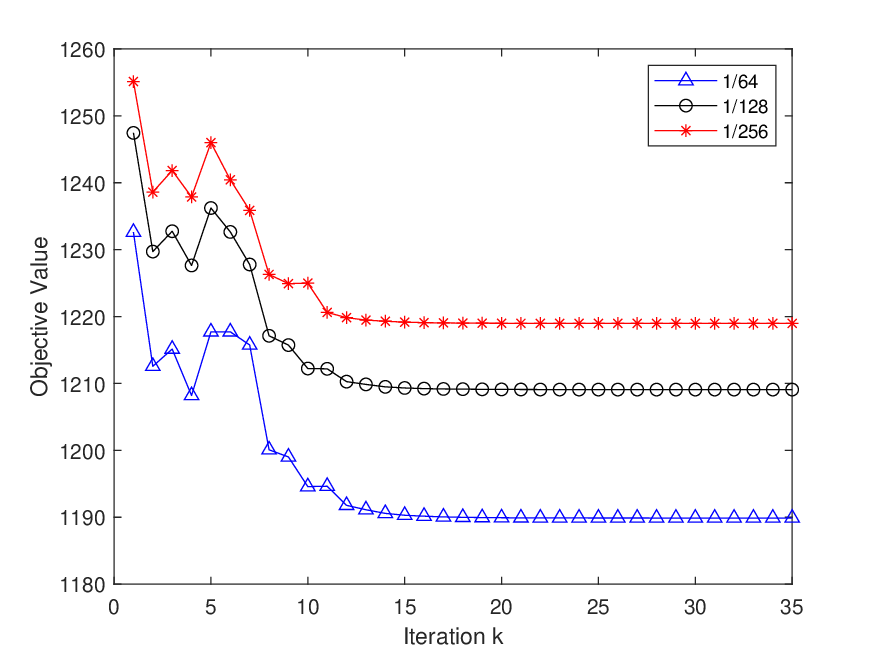}}
		
		\caption{PR, DR, SRD and Obj with different mesh in Example 2.}
		\label{fig:example1case2 obj rel}
	\end{figure}

	
	\subsection{Numerical results for the sparse case}\label{sec:sparse case}
	
	In this subsection, we test our method on the problems with sparse control inputs, i.e., $\gamma_{s} > 0$, including examples where the control is active on the entire domain or only on a subdomain.
	
	We consider the following problem:
	\begin{equation*}
		\min\limits_{u}\ J(u)=\dfrac{\gamma_{d}}{2}\diint_{Q}|y(u)-y_{d}|^{2}{\rm d}x{\rm d}t+\dfrac{1}{2}\diint_{G}|u|^{2}{\rm d}x{\rm d}t+\gamma_{s} \diint_{G}|u|{\rm d}x{\rm d}t,
	\end{equation*}
	where the spatial domain is $\Omega = (0,1)^{2}$, the time horizon is $T=1$, and the state $y$ satisfies the parabolic equation:
	\begin{equation}\label{eq: general state eq of example 3}
		\left\{
		\begin{aligned}
			&\frac{\partial y}{\partial t}-\Delta y=f+\chi_{G}u,&\quad&\text{\,in} \quad\Omega\times(0,T),\\
			&y=0,&\quad  &\text{on}\quad\Gamma\times (0,T),\\
			&y(0)=\phi,&\quad&\text{\,in} \quad\Omega.
		\end{aligned}
		\right.
	\end{equation}
	
	{\bf Example~3~Entire domain control.}
	We adopt the same settings of $\gamma_{d}$, $f$, $\phi$, $y_{d}$, and control constraint set $\mathcal{C}$ as in Example 1. All experiments are conducted with a mesh size $h=\tau=2^{-6}$, using the inexact criterion  $\theta_{k}=\theta_{0}/k^3$ or $\theta_k=\theta_0/2^k$ with $\theta_{0} = \|\sigma_{0}(u^{0})\|/2$.
	
	For different values of the sparsity parameter
	$\gamma_s$, Figure \ref{fig:L1example1 u solution} and Figure \ref{fig:L1example1 u solution under} plot the three-dimensional profiles and the corresponding top views of the control solutions with $\theta_k=\theta_0/2^k$, respectively. Clearly, a larger value of $\gamma_s$ yields a sparser optimal distributed control.
	
	\begin{figure}[H]
		\centering
		
		\subfigure[$\gamma_{s}=0.1$]{\includegraphics[width=0.25\linewidth]{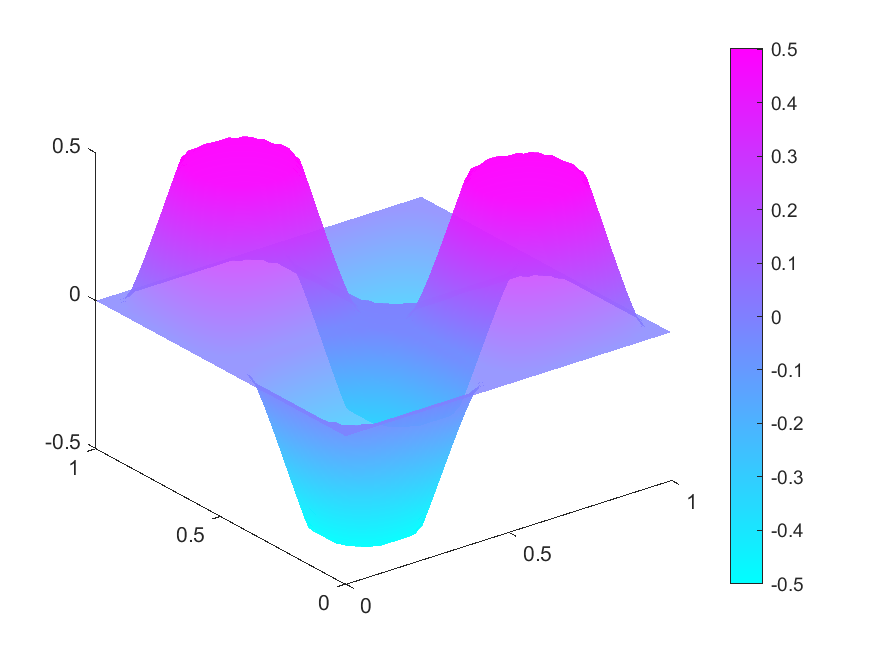}}\hfill
		\subfigure[$\gamma_{s}=0.5$]{\includegraphics[width=0.25\linewidth]{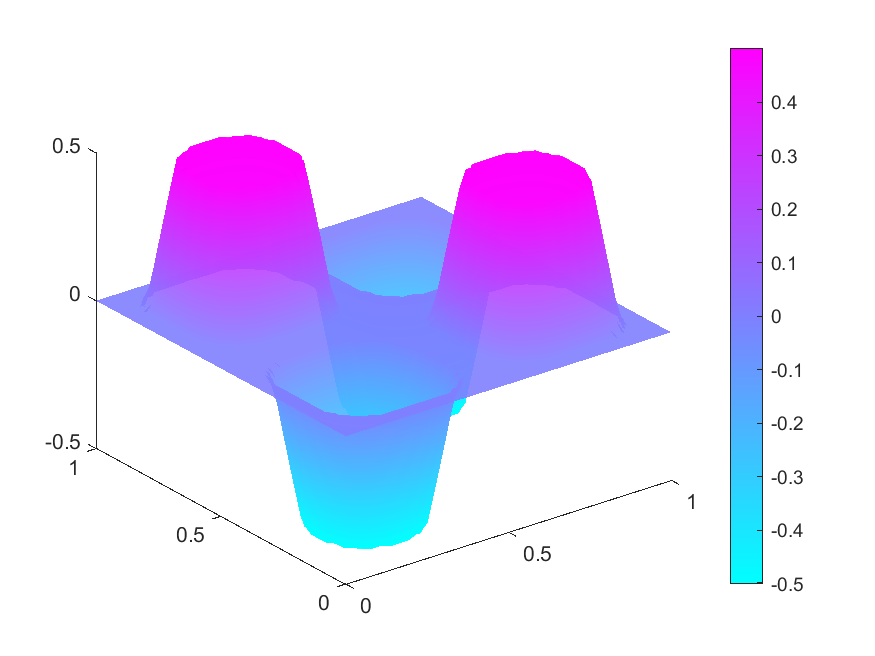}}\hfill
		\subfigure[$\gamma_{s}=5$]{\includegraphics[width=0.25\linewidth]{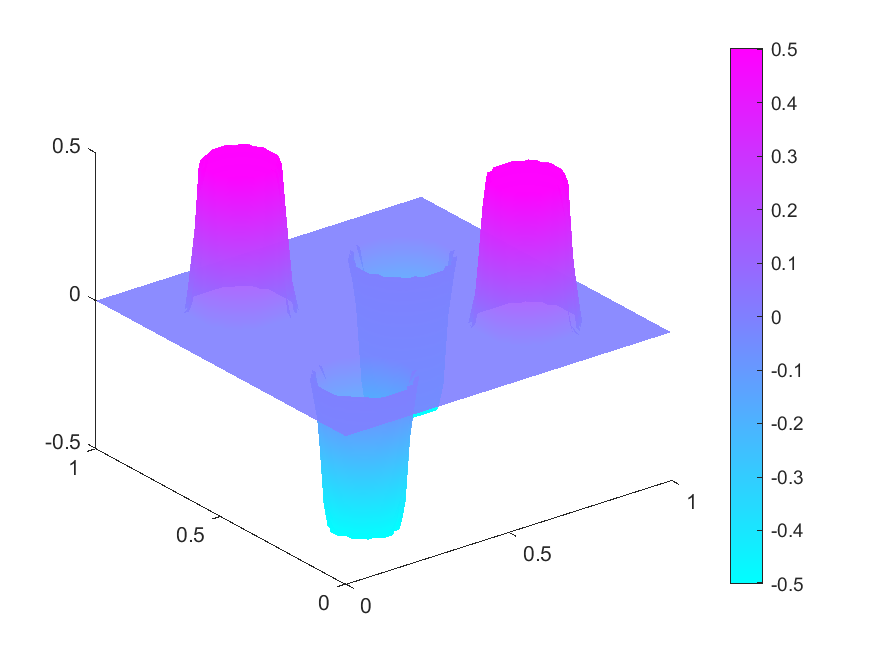}}\hfill
		\subfigure[$\gamma_{s}=10$]{\includegraphics[width=0.25\linewidth]{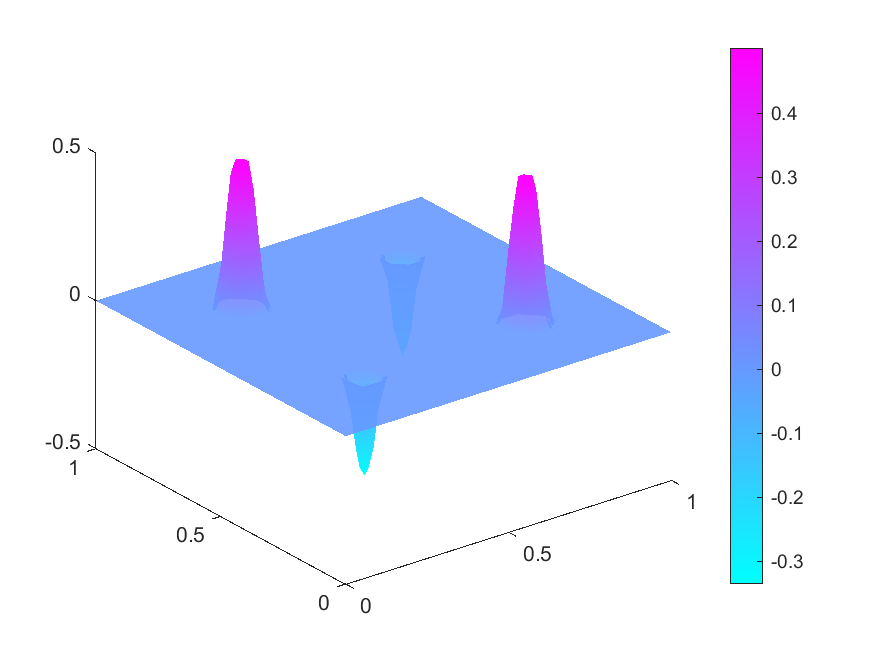}}
		
		\caption{Numerical solutions of $u$ at $t=0.25$ with different $\gamma_{s}$ in Example 3.}
		\label{fig:L1example1 u solution}
	\end{figure}
	
	\begin{figure}[H]
		\centering
		
		\subfigure[$\gamma_{s}=0.1$]{\includegraphics[width=0.23\linewidth]{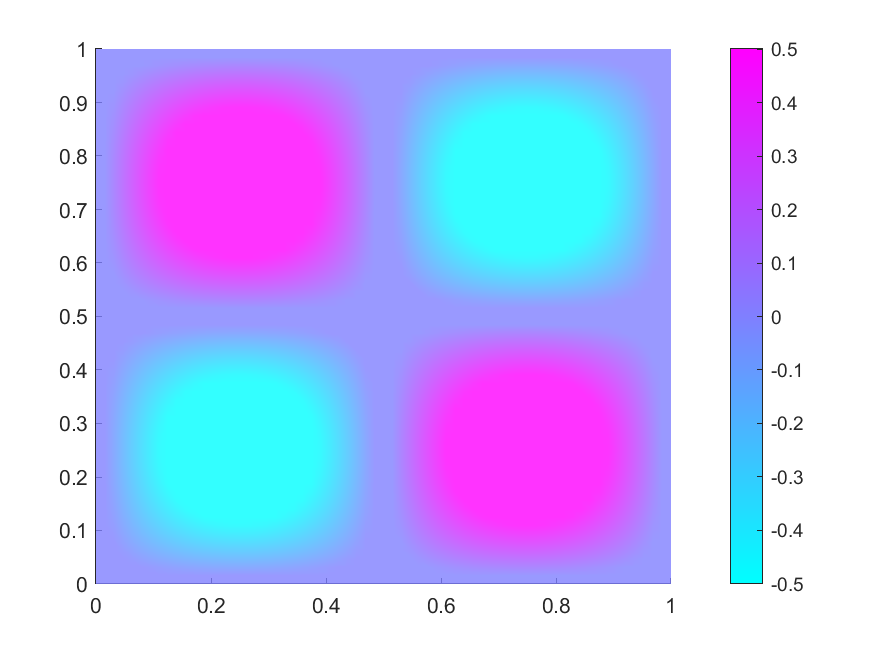}}\hfill
		\subfigure[$\gamma_{s}=0.5$]{\includegraphics[width=0.23\linewidth]{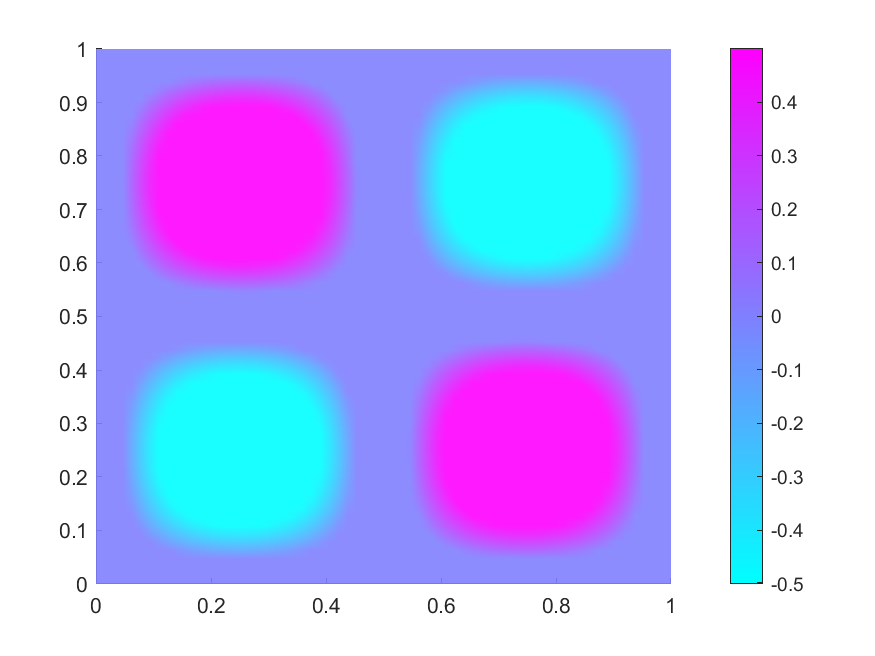}}\hfill
		\subfigure[$\gamma_{s}=5$]{\includegraphics[width=0.23\linewidth]{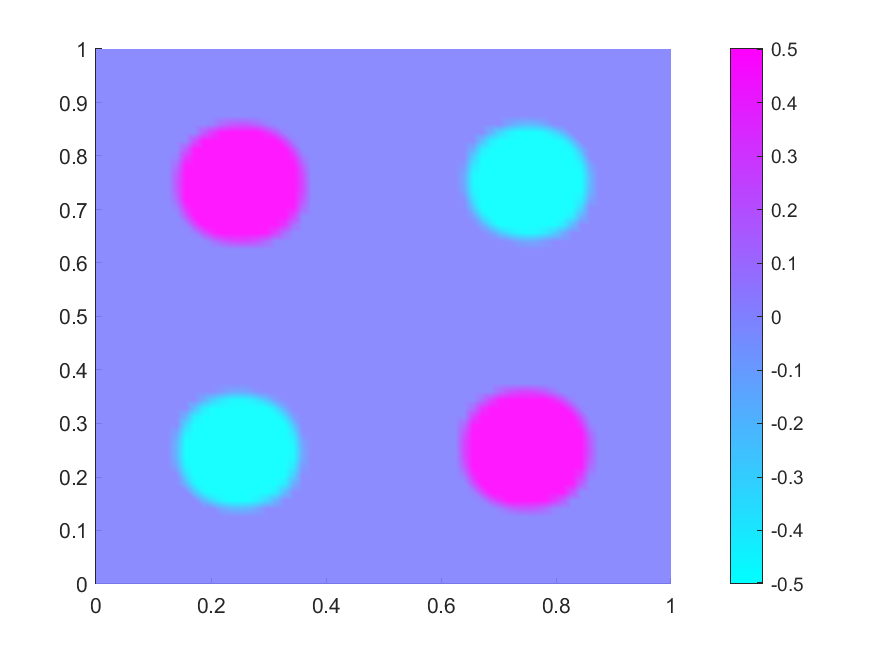}}\hfill
		\subfigure[$\gamma_{s}=10$]{\includegraphics[width=0.23\linewidth]{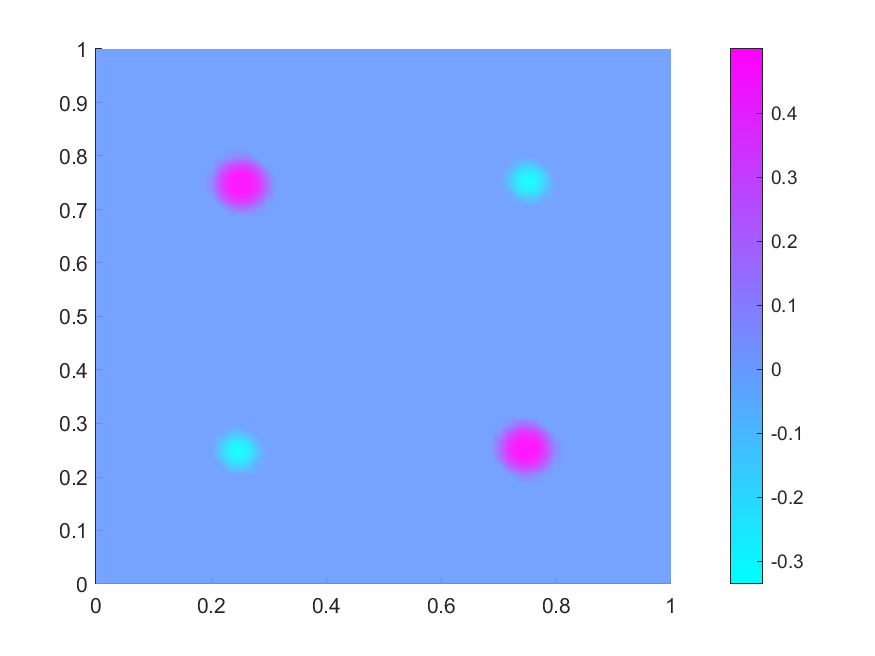}}
		
		\caption{The top-down view of $u$ at $t=0.25$ with different $\gamma_{s}$ in Example 3.}
		\label{fig:L1example1 u solution under}
	\end{figure}
	
	In addition, we compared InADMM and ADMMCG with different values of $\beta_0$ and $\beta_1$. The results summarized in Table~\ref{tab:exp3}. As we can see, varying $\beta_0$ and $\beta_1$ does not affect the solution quality and only influence the computational time. Moreover, the inexact strategy also proves highly efficient for the sparse case, achieving comparable accuracy and significantly reducing the computational time. Similar to Example~1, the inexact criterion with $\theta_{k}=\theta_{0}/k^3$ performs better than that with $\theta_{k}=\theta_{0}/2^k$.
	
	\begin{table}[H]
		\centering
		\caption{Numerical results of InADMM and ADMMCG in Example~3.}
		\scalebox{0.7}{
			\begin{tabular}{|c|c|c|c|c|c|c|c|c|}
				\hline
				Method & $\theta_{k}$ & $(\beta_{0},\beta_{1})$ & $\gamma_{s}$ & Iteration & Time (s) & CG-Ave/Max & Obj & SRD \\
				\hline
				\multirow{12}{*}{InADMM}
				& \multicolumn{1}{c|}{\multirow{4}{*}{$\theta_0/k^3$}} & \multicolumn{1}{c|}{\multirow{2}{*}{$(2,3)$}} &  $0.1$ & $18$ & $7.74$ & $5.50/9$ & $3.43\times 10^{-2}$ & $1.02\times 10^{-3}$ \\
				\cline{4-9}
				& & & $0.5$ & $25$ & $13.64$ & $5.72/8$ & $4.22\times 10^{-2}$ & $1.88\times 10^{-3}$ \\
				\cline{3-9}
				& &\multicolumn{1}{c|}{\multirow{2}{*}{$(10,10)$}} &  $5$ & $65$ & $26.88$ & $5.89/7$ & $2.59\times 10^{-1}$ & $7.74\times 10^{-3}$ \\
				\cline{4-9}
				& & & $10$ & $111$ & $46.86$ & $5.53/7$ & $5.06\times 10^{-1}$ & $1.10\times 10^{-2}$ \\
				\cline{2-9}
				& \multicolumn{1}{c|}{\multirow{4}{*}{$\theta_0/2^k$}} & \multicolumn{1}{c|}{\multirow{2}{*}{$(2,3)$}} &  $0.1$ & $21$ & $17.06$ & $9.29/19$ & $3.43\times 10^{-2}$ & $1.02\times 10^{-3}$ \\
				\cline{4-9}
				& & & $0.5$ & $23$ & $21.51$ & $10.83/21$ & $4.22\times 10^{-2}$ & $1.88\times 10^{-3}$ \\
				\cline{3-9}
				& & \multicolumn{1}{c|}{\multirow{2}{*}{$(10,10)$}} &  $5$ & $66$ & $61.54$ & $11.02/17$ & $2.59\times 10^{-1}$ & $7.74\times 10^{-3}$ \\
				\cline{4-9}
				& & & $10$ & $111$ & $83.86$ & $8.45/16$ & $5.06\times 10^{-1}$ & $1.10\times 10^{-2}$ \\
				\hline
				\multirow{4}{*}{ADMMCG} & \multirow{4}{*}{$ $}& \multirow{4}{*}{$ $} & $0.1$ & $20$ & $42.47$ & $26.00/46$ & $3.43\times 10^{-2}$ & $1.02\times 10^{-3}$ \\
				\cline{4-9}
				& & & $0.5$ & $22$ & $49.96$ & $28.23/46$ & $4.22\times 10^{-2}$ & $1.88\times 10^{-3}$ \\
				\cline{4-9}
				& & & $5$ & $61$ & $85.43$ & $16.36/19$ & $2.59\times 10^{-1}$ & $7.74\times 10^{-3}$ \\
				\cline{4-9}
				& & & $10$ & $108$ & $137.72$ & $14.84/29$ & $5.06\times 10^{-1}$ & $1.10\times 10^{-2}$ \\
				\hline
			\end{tabular}
		}
		\label{tab:exp3}	
	\end{table}

	{\bf Example~4~Subdomain control.}
	In this example, we consider the case that the control takes action on the subdomain $\Omega_{sub} = (0,0.25)^{2}$. The parameter $\gamma_d$, data $f$, $\phi$, $y_d$ and control constraint set $\mathcal{C}$ are set as in Example 2. All experiments are conducted with a mesh size  $h = \tau = 2^{-6}$ and the inexact criterion $\theta_k = \theta_0 / k^3$ or $\theta_k=\theta_0/(1.3)^k$.
	
	We run InADMM for different values of the sparsity parameter $\gamma_s$ and visualize the corresponding numerical control solutions with $\theta_k=\theta_0/(1.3)^k$ in Figure \ref{fig:L1example2 u solution} and Figure \ref{fig:L1example2 u solution under}, which show that larger $\gamma_s$ leads to sparser controls on the subdomain. The numerical results of InADMM and ADMMCG for this example are reported in Table~\ref{tab:exp4}.  For all tested parameters $\beta_{0}$ and $\beta_{1}$, InADMM consistently requires less computational time than ADMMCG to achieve solutions of comparable quality. Moreover, the inexact criterion  $\theta_k = \theta_0 / (1.3)^k$ outperforms that with $\theta_k = \theta_0 / k^3$ in this example.
	
	\begin{figure}[H]
		\centering
		
		\subfigure[$\gamma_{s}=1$]{\includegraphics[width=0.25\linewidth]{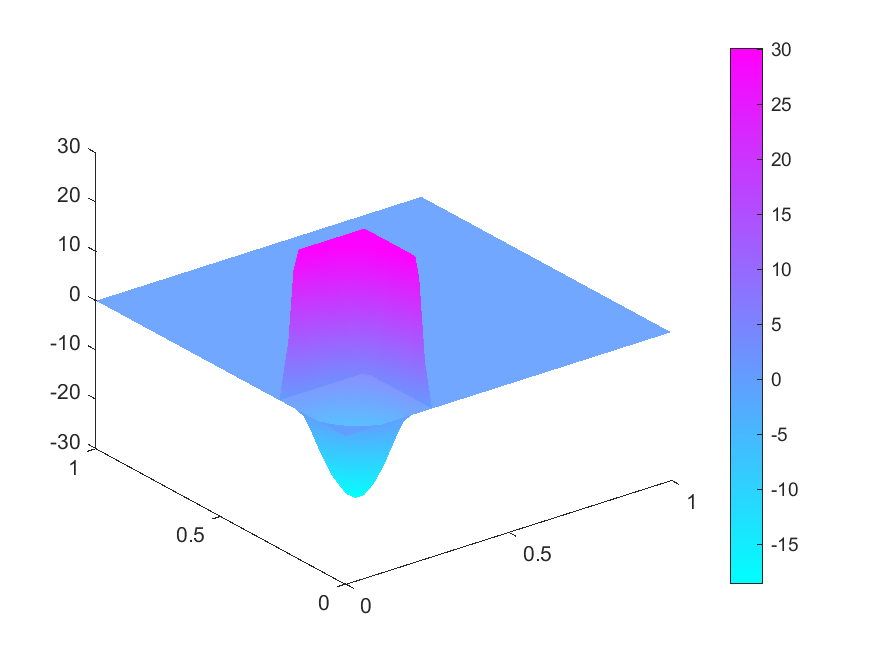}}\hfill
		\subfigure[$\gamma_{s}=10$]{\includegraphics[width=0.25\linewidth]{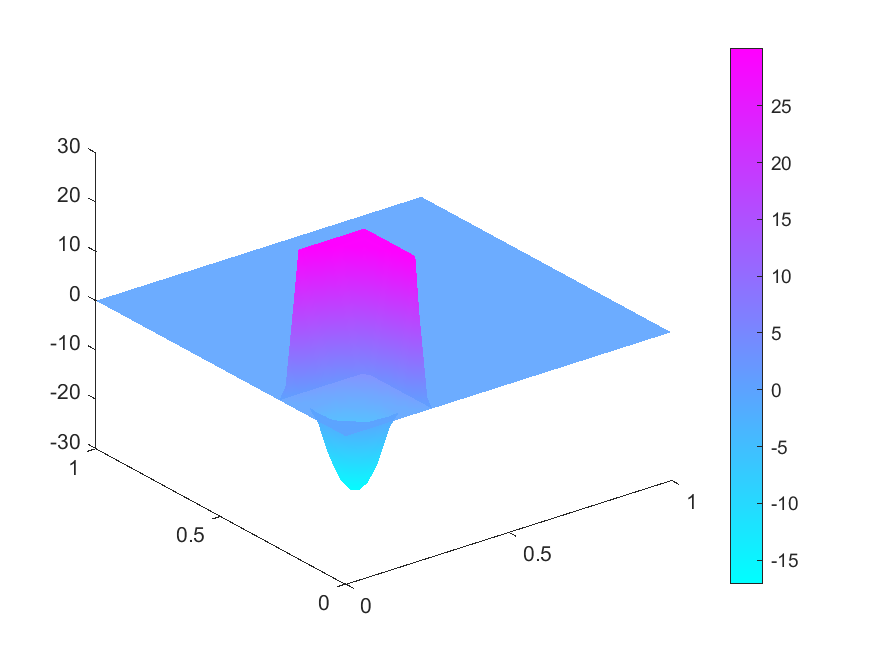}}\hfill
		\subfigure[$\gamma_{s}=50$]{\includegraphics[width=0.25\linewidth]{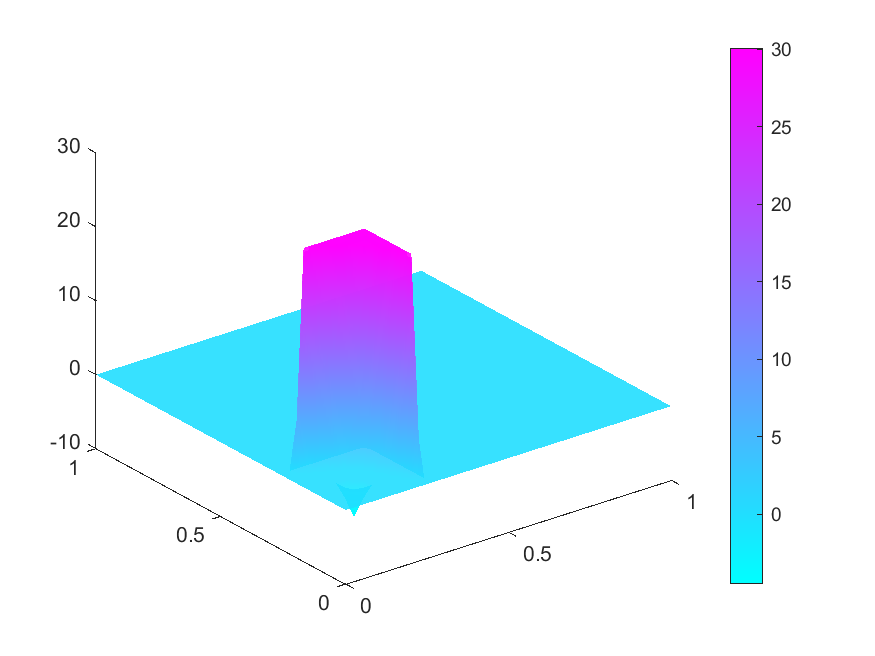}}\hfill
		\subfigure[$\gamma_{s}=500$]{\includegraphics[width=0.25\linewidth]{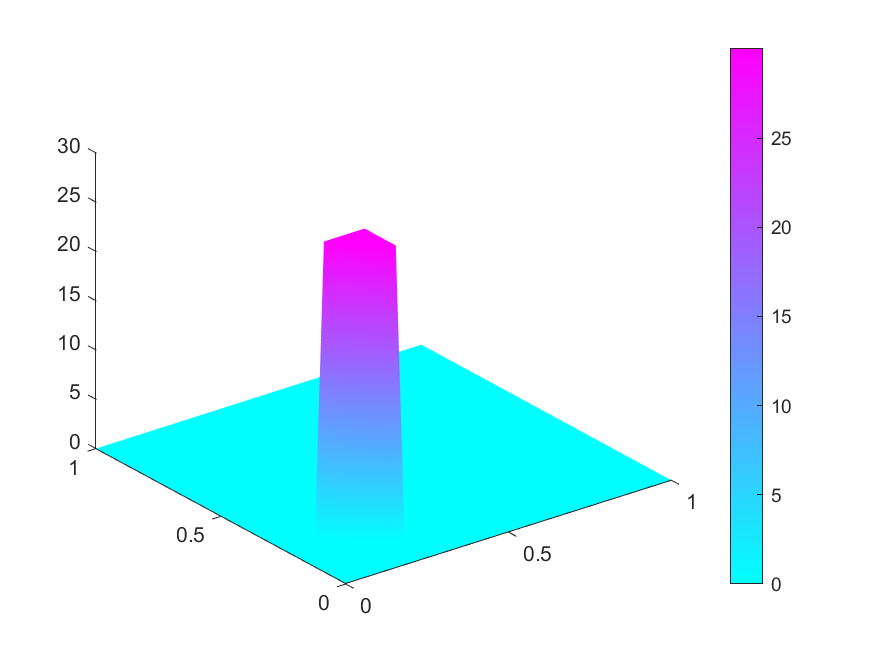}}
		
		\caption{Numerical solutions of $u$ at $t=0.5$ with different $\gamma_{s}$ in Example 4.}
		\label{fig:L1example2 u solution}
	\end{figure}
	
	\begin{figure}[H]
		\centering
		
		\subfigure[$\gamma_{s}=1$]{\includegraphics[width=0.23\linewidth]{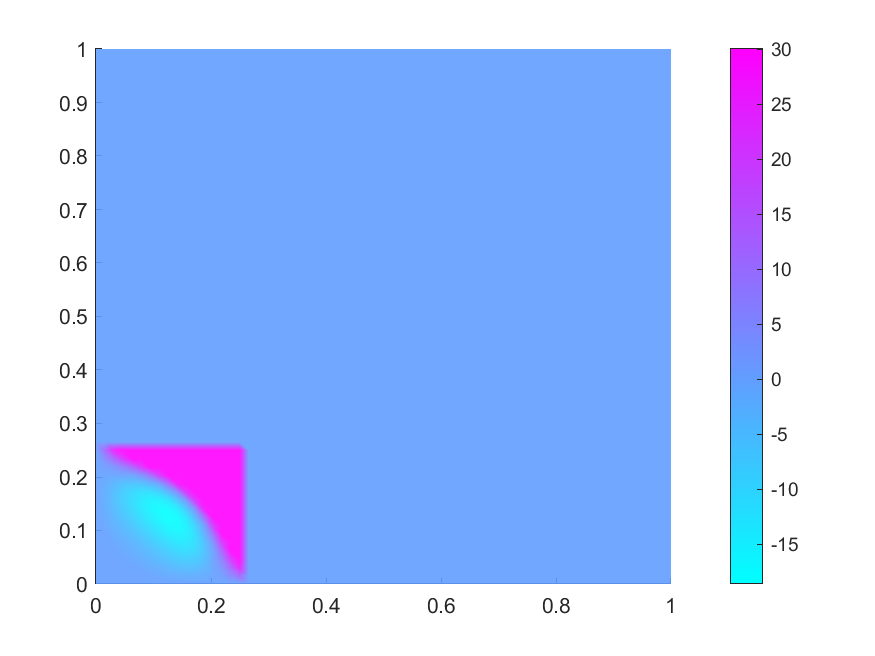}}\hfill
		\subfigure[$\gamma_{s}=10$]{\includegraphics[width=0.23\linewidth]{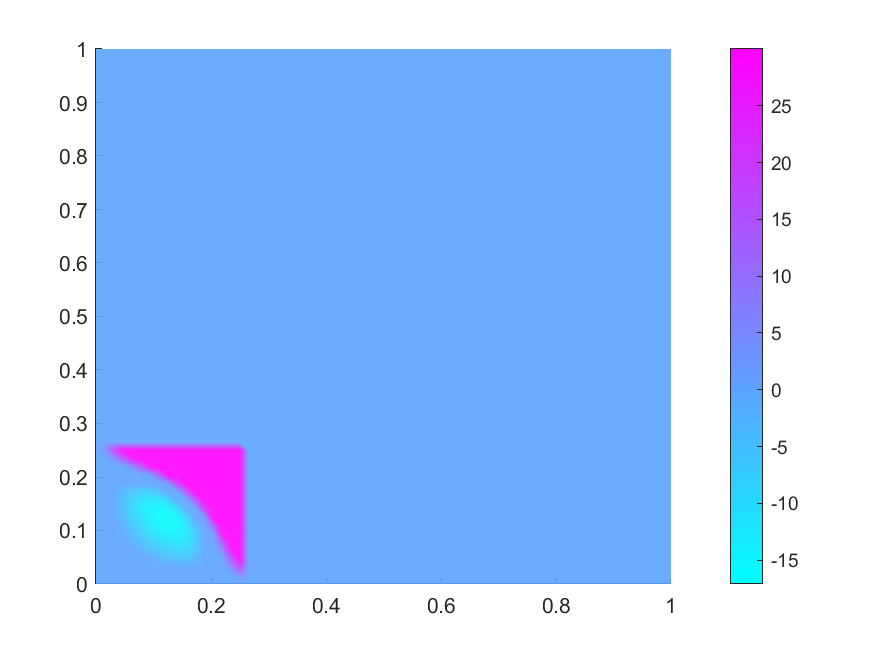}}\hfill
		\subfigure[$\gamma_{s}=50$]{\includegraphics[width=0.23\linewidth]{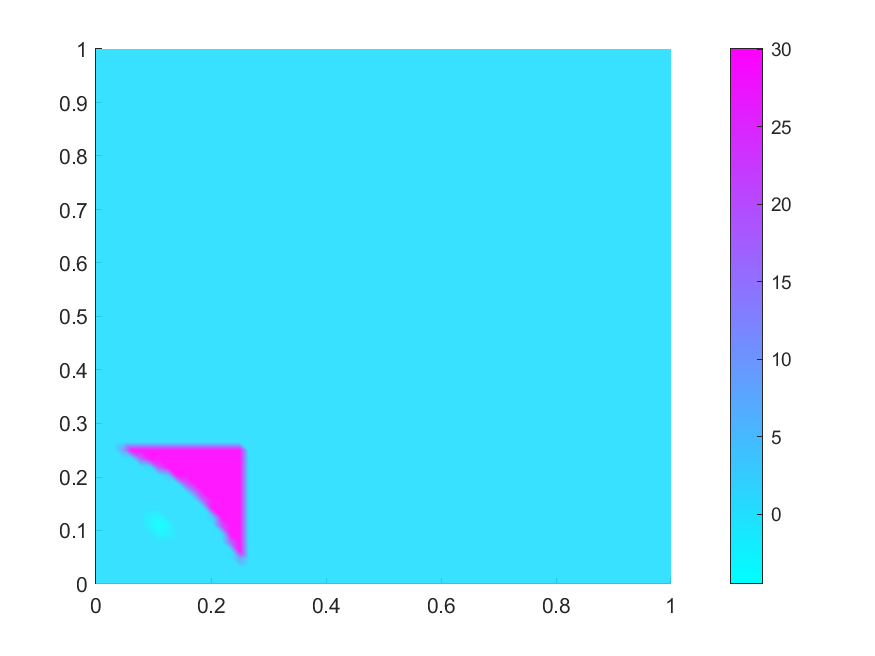}}\hfill
		\subfigure[$\gamma_{s}=500$]{\includegraphics[width=0.23\linewidth]{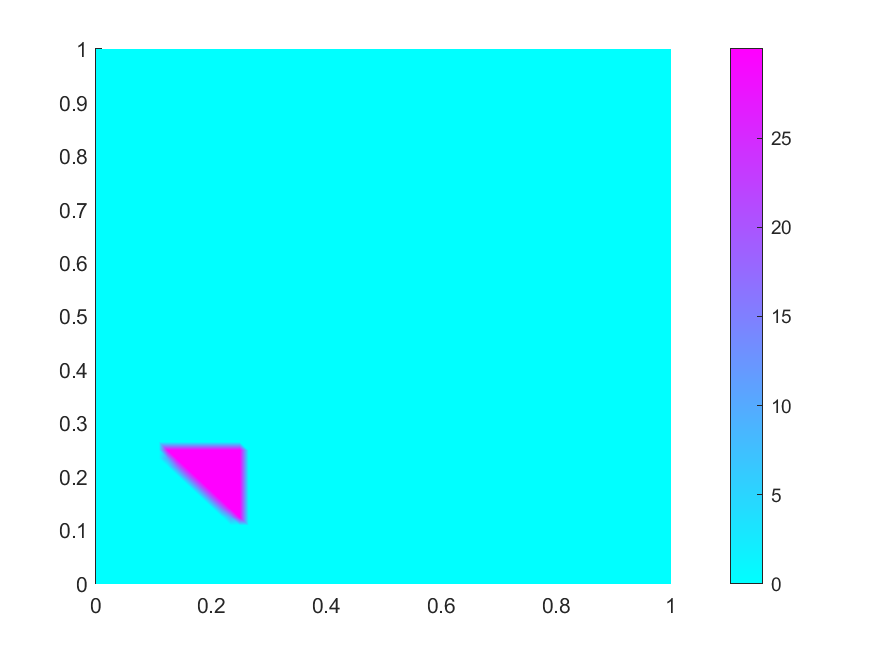}}
		
		\caption{The top-down view of $u$ at $t=0.5$ with different $\gamma_{s}$ in Example 4.}
		\label{fig:L1example2 u solution under}
	\end{figure}

	\begin{table}[H]
		\centering
		\caption{Numerical results of InADMM and ADMMCG in Example~4.}
		\scalebox{0.7}{
			\begin{tabular}{|c|c|c|c|c|c|c|c|c|}
				\hline
				Method & $\theta_{k}$ & $(\beta_{0},\beta_{1})$ & $\gamma_{s}$ & Iteration & Time (s) & CG-Ave/Max & Obj ($\times 10^{3}$) & SRD ($\times 10^{-1}$) \\
				\hline
				\multirow{12}{*}{InADMM}
				& \multicolumn{1}{c|}{\multirow{4}{*}{$\theta_0/k^3$}} & \multicolumn{1}{c|}{\multirow{2}{*}{$(8,8)$}} &  $1$ & $34$ & $7.02$ & $4.62/8$ & $1.19$ & $8.16$ \\
				\cline{4-9}
				& & & $10$ & $33$ & $6.76$ & $4.70/8$ & $1.19$ & $8.16$ \\
				\cline{3-9}
				& & \multicolumn{1}{c|}{\multirow{2}{*}{$(10,10)$}} & $50$ & $37$ & $6.15$ & $3.84/6$ & $1.20$ & $8.19$ \\
				\cline{4-9}
				& & & $500$ & $46$ & $6.72$ & $3.28/4$ & $1.25$ & $8.36$ \\
				\cline{2-9}
				& \multicolumn{1}{c|}{\multirow{4}{*}{$\theta_0/(1.3)^k$}} & \multicolumn{1}{c|}{\multirow{2}{*}{$(8,8)$}} &  $1$ & $44$ & $4.65 $ & $2.16/4$ & $1.19$ & $8.16$ \\
				\cline{4-9}
				& & & $10$ & $46$ & $5.26$ & $2.17/4$ & $1.19$ & $8.16$ \\
				\cline{3-9}
				& & \multicolumn{1}{c|}{\multirow{2}{*}{$(10,10)$}} & $50$ & $49$ & $5.27 $ & $2.18/4$ & $1.20$ & $8.19$ \\
				\cline{4-9}
				& & & $500$ & $57$ & $6.69 $ & $2.44/5$ & $1.25$ & $8.36$ \\
				\hline
				\multirow{4}{*}{ADMMCG} & \multirow{4}{*}{$ $}& \multirow{4}{*}{$ $} &  $1$ & $33$ & $26.82$ & $20.00/46$ & $1.19$ & $8.16$ \\
				\cline{4-9}
				& & & $10$ & $32$ & $25.73$ & $20.13/46$ & $1.19$ & $8.16$ \\
				\cline{4-9}
				& & & $50$ & $30$ & $23.09$ & $19.13/41$ & $1.20$ & $8.19$ \\
				\cline{4-9}
				& & & $500$ & $35$ & $25.39$ & $17.97/31$ & $1.25$ & $8.36$ \\
				\hline
			\end{tabular}
		}
		\label{tab:exp4}	
	\end{table}

	\section{Conclusions}\label{sec:conclusion}
	We have proposed an inexact algorithmic framework of the alternating direction method of multipliers for solving constrained parabolic optimal distributed control problem. At each iteration, the nonsmooth part of objective and control constraint are untied from the parabolic optimal control problem, and the resulting parabolic equation constrained subproblem is solved inexactly using our efficient strategy. With the flexible parameter choices, we have established the globally strong convergence and a linear convergence rate of the proposed method. A key practical advantage of our inexact framework is that its implementation  undemanding, requiring only a few internal iterations per outer step. Numerical experiments have illustrated the validness and competitive efficiency of our methods.

	\section*{Declarations}
	{\bf Funding~}The work was supported by National Science and Technology Major Project for Deep Earth probe and Mineral Resources Exploration-National Science and Technology Major Project under Grants 2025ZD1008500 and 2025ZD1008503, Research Project of the Education Department of Jilin Province under Grant JJKH20262221BS,  and Graduate Innovation Fund of Jilin University under Grants 2024KC036, 2025CX090 and 2025CX097.
	
	\vskip4mm
	\noindent
	{\bf Data availibility~} The work does not analyse or generate any datasets.
	
	\vskip4mm
	\noindent
	{\bf Competing interests~}The authors have no financial or proprietary interests in any material discussed in this paper.


	%
	%

	\bibliographystyle{plain}      
	\bibliography{ref}   

\begin{thebibliography}{10}

\bibitem{baraldi2022proximal}
Robert~J. Baraldi and Drew~P. Kouri.
\newblock A proximal trust-region method for nonsmooth optimization with
  inexact function and gradient evaluations.
\newblock {\em Mathematical Programming}, 201(1):559--598, 2022.

\bibitem{bergounioux1999primal}
Ma{\"\i}tine Bergounioux, Kazufumi Ito, and Karl Kunisch.
\newblock Primal-dual strategy for constrained optimal control problems.
\newblock {\em SIAM Journal on Control and Optimization}, 37(4):1176--1194,
  1999.

\bibitem{bersetche2025fractional}
Francisco Bersetche, Francisco Fuica, Enrique Ot{\'a}rola, and Daniel Quero.
\newblock Fractional, semilinear, and sparse optimal control: A priori error
  bounds.
\newblock {\em Applied Mathematics and Optimization}, 91:20, 2025.

\bibitem{borzi2003multigrid}
Alfio Borz{\`\i}.
\newblock Multigrid methods for parabolic distributed optimal control problems.
\newblock {\em Journal of Computational and Applied Mathematics},
  157(2):365--382, 2003.

\bibitem{boyd2004convex}
Stephen Boyd and Lieven Vandenberghe.
\newblock {\em Convex Optimization}.
\newblock Cambridge University Press, 2004.

\bibitem{casas2024convergence}
Eduardo Casas and Mariano Mateos.
\newblock Convergence analysis of the semismooth {Newton} method for sparse
  control problems governed by semilinear elliptic equations.
\newblock {\em SIAM Journal on Control and Optimization}, 62(6):3076--3090,
  2024.

\bibitem{chang2022sparse}
Lili Chang, Wei Gong, Zhen Jin, and Gui-Quan Sun.
\newblock Sparse optimal control of pattern formations for an {SIR}
  reaction-diffusion epidemic model.
\newblock {\em SIAM Journal on Applied Mathematics}, 82(5):1764--1790, 2022.

\bibitem{colli2022optimal}
Pierluigi Colli, Andrea Signori, and J{\"u}rgen Sprekels.
\newblock Optimal control problems with sparsity for tumor growth models
  involving variational inequalities.
\newblock {\em Journal of Optimization Theory and Applications}, 194(1):25--58,
  2022.

\bibitem{de2015numerical}
Juan~Carlos De~los Reyes.
\newblock {\em Numerical {PDE}-Constrained Optimization}.
\newblock Springer, Cham, Switzerland, 2015.

\bibitem{gander2016schwarz}
Martin~J. Gander and Felix Kwok.
\newblock Schwarz methods for the time-parallel solution of parabolic control
  problems.
\newblock In {\em Domain Decomposition Methods in Science and Engineering
  XXII}, pages 207--216. Springer, 2016.

\bibitem{garcke2021sparse}
Harald Garcke, Kei~Fong Lam, and Andrea Signori.
\newblock Sparse optimal control of a phase field tumor model with mechanical
  effects.
\newblock {\em SIAM Journal on Control and Optimization}, 59(2):1555--1580,
  2021.

\bibitem{glowinski2008exact}
Roland Glowinski, Jacques-Louis Lions, and Jiwen He.
\newblock {\em Exact and Approximate Controllability for Distributed Parameter
  Systems: A Numerical Approach}.
\newblock Encyclopedia of Mathematics and its Applications. Cambridge
  University Press, 2008.

\bibitem{glowinski2022application}
Roland Glowinski, Yongcun Song, Xiaoming Yuan, and Hangrui Yue.
\newblock Application of the alternating direction method of multipliers to
  control constrained parabolic optimal control problems and beyond.
\newblock {\em Annals of Applied Mathematics}, 38(2):115--158, 2022.

\bibitem{he2002new}
Bingsheng He, Li-Zhi Liao, Deren Han, and Hai Yang.
\newblock A new inexact alternating directions method for monotone variational
  inequalities.
\newblock {\em Mathematical Programming}, 92:103--118, 2002.

\bibitem{herzog2010algorithms}
Roland Herzog and Karl Kunisch.
\newblock Algorithms for pde-constrained optimization.
\newblock {\em GAMM-Mitteilungen}, 33(2):163--176, 2010.

\bibitem{hintermuller2010semismooth}
Michael Hinterm\"uller.
\newblock Semismooth newton methods and applications.
\newblock Technical report, Department of Mathematics, Humboldt University of
  Berlin, 2010.

\bibitem{hintermuller2002primal}
Michael Hinterm{\"u}ller, Kazufumi Ito, and Karl Kunisch.
\newblock The primal-dual active set strategy as a semismooth newton method.
\newblock {\em SIAM Journal on Optimization}, 13(3):865--888, 2002.

\bibitem{hinze2009optimization}
Michael Hinze, Ren{\'e} Pinnau, Michael Ulbrich, and Stefan Ulbrich.
\newblock {\em Optimization with PDE Constraints}, volume~23 of {\em
  Mathematical Modelling: Theory and Applications}.
\newblock Springer, Dordrecht, 2009.

\bibitem{lin2022alternating}
Zhouchen Lin, Huan Li, and Cong Fang.
\newblock {\em Alternating Direction Method of Multipliers for Machine
  Learning}.
\newblock Springer, Singapore, 2022.

\bibitem{lions1971optimal}
Jacques~Louis Lions.
\newblock {\em Optimal control of systems governed by partial differential
  equations}, volume 170 of {\em Applied Mathematical Sciences}.
\newblock Springer-Verlag, Berlin, Heidelberg, New York, 1971.

\bibitem{mathew2010analysis}
Tarek~P. Mathew, Marcus Sarkis, and Christian~E. Schaerer.
\newblock Analysis of block parareal preconditioners for parabolic optimal
  control problems.
\newblock {\em SIAM Journal on Scientific Computing}, 32(3):1180--1200, 2010.

\bibitem{mcdonald2016all}
Eleanor McDonald.
\newblock {\em All-at-once solution of time-dependent PDE problems}.
\newblock PhD thesis, University of Oxford, 2016.

\bibitem{nesterov2004introductory}
Yurii Nesterov.
\newblock {\em Introductory Lectures on Convex Optimization: A Basic Course},
  volume~87 of {\em Applied Optimization}.
\newblock Springer, Boston, MA, 2004.

\bibitem{nocedal2006numerical}
Jorge Nocedal and Stephen~J. Wright.
\newblock {\em Numerical Optimization}.
\newblock Springer Series in Operations Research and Financial Engineering.
  Springer, New York, NY, 2 edition, 2006.

\bibitem{pearson2012regularization}
John~W. Pearson, Martin Stoll, and Andrew~J. Wathen.
\newblock Regularization-robust preconditioners for time-dependent
  pde-constrained optimization problems.
\newblock {\em SIAM Journal on Matrix Analysis and Applications},
  33(4):1126--1152, 2012.

\bibitem{porcelli2015preconditioning}
Margherita Porcelli, Valeria Simoncini, and Mattia Tani.
\newblock Preconditioning of active-set newton methods for pde-constrained
  optimal control problems.
\newblock {\em SIAM Journal on Scientific Computing}, 37(5):S472--S502, 2015.

\bibitem{rees2010optimal}
Tyrone Rees, H~Sue Dollar, and Andrew~J Wathen.
\newblock Optimal solvers for pde-constrained optimization.
\newblock {\em SIAM Journal on Scientific Computing}, 32(1):271--298, 2010.

\bibitem{schiela2014operator}
Anton Schiela and Stefan Ulbrich.
\newblock Operator preconditioning for a class of inequality constrained
  optimal control problems.
\newblock {\em SIAM Journal on Optimization}, 24(1):435--466, 2014.

\bibitem{song2023numerical}
Xiaoliang Song and Bo~Yu.
\newblock Numerical solution for sparse {PDE} constrained optimization.
\newblock In {\em Handbook of Mathematical Models and Algorithms in Computer
  Vision and Imaging: Mathematical Imaging and Vision}, pages 623--675.
  Springer, Cham, 2023.

\bibitem{song2024admm}
Yongcun Song, Xiaoming Yuan, and Hangrui Yue.
\newblock The {ADMM-PINNs} algorithmic framework for nonsmooth
  {PDE}-constrained optimization: A deep learning approach.
\newblock {\em SIAM Journal on Scientific Computing}, 46(6):C659--C687, 2024.

\bibitem{sprekels2021sparse}
J{\"u}rgen Sprekels and Fredi Tr{\"o}ltzsch.
\newblock Sparse optimal control of a phase field system with singular
  potentials arising in the modeling of tumor growth.
\newblock {\em ESAIM: Control, Optimisation and Calculus of Variations},
  27:S26, 2021.

\bibitem{stoll2014one}
Martin Stoll.
\newblock One-shot solution of a time-dependent time-periodic pde-constrained
  optimization problem.
\newblock {\em IMA Journal of Numerical Analysis}, 34(4):1554--1577, 2014.

\bibitem{stoll2012preconditioning}
Martin Stoll and Andy Wathen.
\newblock Preconditioning for partial differential equation constrained
  optimization with control constraints.
\newblock {\em Numerical Linear Algebra with Applications}, 19(1):53--71, 2012.

\bibitem{troeltzsch2010optimal}
Fredi Tr{\"o}ltzsch.
\newblock {\em Optimal Control of Partial Differential Equations: Theory,
  Methods and Applications}, volume 112 of {\em Graduate Studies in
  Mathematics}.
\newblock American Mathematical Society, Providence, RI, 2010.

\bibitem{ulbrich2007generalized}
Stefan Ulbrich.
\newblock Generalized sqp methods with “parareal” time-domain decomposition
  for time-dependent pde-constrained optimization.
\newblock In {\em Real-time PDE-constrained optimization}, pages 145--168.
  SIAM, 2007.

\bibitem{wang2025adaptive}
Fangyuan Wang, Qiming Wang, and Zhaojie Zhou.
\newblock Adaptive finite element approximation of sparse optimal control
  problem with integral fractional laplacian.
\newblock {\em Journal of Scientific Computing}, 102(1):17, 2025.

\bibitem{Wang2023Duality}
Hailing Wang, Di~Wu, Changjun Yu, and Kok~Lay Teo.
\newblock A duality-based approach for linear parabolic optimal control
  problems.
\newblock {\em Optimal Control Applications and Methods}, 45(3):1140--1165,
  2023.

\bibitem{xu2021iteration}
Yangyang Xu.
\newblock Iteration complexity of inexact augmented lagrangian methods for
  constrained convex programming.
\newblock {\em Mathematical Programming}, 185(1):199--244, 2021.

\end{thebibliography}
	

\end{document}